\documentclass[a4paper,12pt,oneside,reqno]{amsart}
\usepackage{hyperref}
\usepackage[headinclude,DIV13]{typearea}
\areaset{15.1cm}{25.0cm}
\usepackage{amsfonts,amssymb,amsmath,amsthm,bbm}
\usepackage[latin1] {inputenc}
\usepackage{graphicx, psfrag}

\usepackage{float}

\usepackage{bold-extra}

\usepackage{amsmath}%
\usepackage{amsfonts}%
\usepackage{amssymb}%
\usepackage{graphicx}
\usepackage{srcltx}%
\usepackage{dsfont}
\usepackage[latin1]{inputenc}
\usepackage{graphicx}
\usepackage{color}
\usepackage{caption}
\numberwithin{equation}{section}

\newtheorem{theorem}{Theorem}[section]

\newtheorem{lemma}[theorem]{Lemma}
\newtheorem{proposition}[theorem]{Proposition}
\newtheorem{corollary}[theorem]{Corollary}

\theoremstyle{definition}

\theoremstyle{remark}
\newtheorem{remark}{Remark}[section]

\theoremstyle{plain}

\numberwithin{equation}{section}

\newcommand{\E}{\mathbb{E}}
\newcommand{\N}{\mathbb{N}}
\newcommand{\nn}{{(n)}}
\newcommand{\Pp}{\mathcal{P}}
\newcommand{\pp}{\mathbb{P}}

\newcommand{\qw}{\begin{equation}}
\newcommand{\qwe}{\end{equation}}
\newcommand{\qww}{\begin{equation*}}
\newcommand{\qwee}{\end{equation*}}

\newcommand{\aan}[1]{{a_{n}(#1)}}
\newcommand{\ccn}[1]{{c_{n}(#1)}}
\newcommand{\mmn}[1]{n}
\newcommand{\tn}{c_n}

\newcommand{\inc}[1]{\mathds{1}\left\{{#1}\right\}}

\newcommand{\sll}[2]{S_{{#1},{#2}}}

\newcommand{\tre}{\mathbb{T}_L}

\newcommand{\lag}{{l^*}}

\newcommand{\luc}{{\lfloor u \ccn{k} \rfloor}}

\newcommand{\Ee}{\mathcal{E}}

\newcommand{\wt}{\widetilde}
\newcommand{\bv}[1]{\bar{\nu}^\nn_{#1}(u;\mu|_{#1-1})}

\newcommand{\tred}{\textcolor{red}}

\newcommand{\hideownbib}[1]{} %#1

\hideownbib{\newcites{own}{[a] Refereed}}
\hideownbib{\bibliographystyleown{plain}}

\begin{document}
\title[Aging in the GREM-like trap model]{Aging in the GREM-like trap model}

\author{V\'eronique Gayrard, Onur G\"{u}n}
\address
{V\'eronique Gayrard, I2M, Aix-Marseille Universit\'e \newline
\indent 39 rue F. Joliot Curie\newline
\indent 13453 Marseille
cedex 13, FRANCE}
\email{veronique.gayrard@math.cnrs.fr, veronique@gayrard.net}
\address
{Onur G\"{u}n, Weierstrass Institute Berlin\newline
\indent 39 Mohrenstrasse\newline
\indent 10117 Berlin, GERMANY}%
\email{guen@wias-berlin.de}%

\date{\today}
%Dynamics of disordered systems, spin glasses.

\begin{abstract}
The GREM-like trap model is a continuous time Markov jump process on the leaves of a finite volume $L$-level tree whose transition rates depend on a \emph{trapping landscape} built on the vertices of the whole tree. We prove that the natural two-time correlation function of the dynamics ages in the infinite volume limit and identify the limiting function. Moreover, we take the limit $L\to\infty$ of the two-time correlation function of the infinite volume $L$-level tree. The aging behavior of the dynamics is characterized by a collection of \emph{clock processes}, one for each level of the tree.\,We show that for any $L$, the joint law of the clock processes converges. Furthermore, any such limit can be expressed through Neveu's continuous state branching process. Hence, the latter contains all the information needed to describe aging in the GREM-like trap model both for finite and infinite levels.

\end{abstract}
\maketitle

\section{Introduction}\label{sc1}

Trap models are main theoretical tools to quantify the out-of-equilibrium dynamics of spin glasses, and more specifically their aging behavior (see \cite{BCKM98} for a review). Introduced in this context by J.P.~Bouchaud \cite{Bou92} to model the dynamics of mean field spin glasses such as the REM, GREM, and $p$-spin SK models, trap models are simple Markov jump processes that describe dynamics on microscopic spin space in terms of thermally activated barrier crossing between the valleys (or \emph{traps}) of a random landscape on reduced state space devised
%defined 
so as to retain some of the key features of the free energy landscape of the underlying
% spin glass model.
spin system.
Activated aging occurs if, on time scales that diverge with the size of the system, the process 
observed through 
%using
suitably chosen  time-time correlation functions becomes slower and slower as time elapses.
%(Activated aging occurs if, on time-scales that diverge with the size of the system, the process 
%observed through 
%%using
%suitably chosen  time-time correlation functions
%is abnormally slow, namely, if the de-correlation becomes slower and slower as time elapses.)

%Activated aging occurs if, on time scales that diverge with the size of the system, the process 
%becomes slower and slower as time elapses, as observed through suitably chosen  time-time correlation functions.

%a pattern of behavior quantified/observed using suitable time-time correlation functions.

%The first rigorous connection between the microscopic dynamics of a spin system and a trap model was established in \cite{} where it is proved that a particular Glauber dynamics of the REM, known as the Random Hopping Dynamics (hereafter RHD), observed on ``the last time scale before stationarity'' has the same aging behavior as Bouchaud?s symmetric trap model on the complete graph \cite{}. This result was followed up by a series of articles yielding a detailed understanding of the aging behavior of the RHD dynamics of the REM  \cite{} (and of its Metropolis dynamics)
%

The first rigorous connection between the microscopic dynamics of a spin system and a trap model was established in \cite{BBG03a,BBG03b} where it is proved that a particular Glauber dynamics of the REM has the same aging behavior as Bouchaud's symmetric trap model on the complete graph or ``REM-like trap model'' \cite{BD95}. This result was followed up by a series of results yielding a detailed understanding of the aging behavior of the REM for a wide range of time scales and temperatures, and various dynamics \cite{BC08,Ver2,VerMetro}. Spin glasses with non-trivial correlations, namely the $p$-spin SK models, could also be dealt with albeit in a restricted domain of the time scales and temperature parameters, where it was proved that aging is just as in the REM  \cite{BBC08,BG12,BG13,BGS13}. This reflects the fact that the dynamics is insensitive to the correlation structure of the random environment.
%correlation structure of the random environment does not influence the aging dynamics.
% not sensitive to
Altough we do expect that, on longer time scales, the aging dynamics of the $p$-spin SK models belongs to (a) different universality class(es), there is yet no rigorous result
% in this direction
supporting this idea.
%such a claim

%At a heuristic level, 
%%a first attempt to account for the the effect of strongly  correlated random environment was made by made by Bouchaud and Dean \cite{} who generalize....
%the effect of strongly  correlated random environments
% %the correlation structure 
% on  activated aging was first accounted for by Bouchaud and Dean \cite{} who generalized the REM-like trap model of \cite{} to a dynamics in a  landscape of random traps organized according to a hierarchical tree structure
%%on a foliated tree structure 
%inspired from Parisi's ultrametric construction \cite{}.
%%(
%% that mimics the standard picture of a spin-glass phase  in the ``full replica symmetry breaking'' \cite{}.
%% )

%At a heuristic level, the effect of strongly  correlated random environments on activated aging was first accounted for by 
%Bouchaud and Dean \cite{} using a trap model in a landscape of random traps organized according to a hierarchical tree structure inspired from Parisi's ultrametric construction \cite{}.
%
At a heuristic level, possible effects of strongly correlated random environments on activated aging were first modeled by Bouchaud and Dean \cite{BD95} using a trap model whose random traps are organized according to a hierarchical tree structure inspired by Parisi's ultrametric construction \cite{MPSTV84}.
More recently, Sasaki and Nemoto \cite{SN00a, SN01} introduced a trap model on a tree with a view to modeling the aging dynamics of the GREM. 
%From a mathematical perspective their approach seems more promising: indeed
From a mathematical perspective this GREM-like trap model seems more promising. Indeed a detailed (rigorous) analysis of the statics of the GREM (as well as that of a more general class of continuous random energy models, or CREMs)
%(including a description of the limiting Gibbs measures through Ruelle's Poisson cascades)
is available (see \cite{BK07} and the references therein), making it plausible to expect that the GREM-like trap model correctly predicts 
%the main features of 
the behavior of the aging dynamics of the GREM itself, at least  in some domain of the temperature and time scale parameters.
%the predictions based on that model are correct/
%making (it) possible to/
% more plausible/
%it is plausible to expect that the predictions based on that model are correct/
%the GREM-like trap model correctly predicts the aging dynamics of the GREM itself, at least  in some domain of the temperature and time scale parameters.

With this in mind, our aim in this paper is two-fold. Firstly, we want to identify the aging behavior of the GREM-like trap model as the branch size of the tree, $n$, diverges, and also as the number of levels, $L$,  diverges after the limit $n\to\infty$ is taken. Secondly, we want to emphasize that Neveu's continuous state branching process (hereafter abbreviated CSBP) naturally describes aging in the GREM-like trap model, namely, the aging behavior of the dynamics is encoded in a collection of \emph{clock processes} and all possible limits of these clock processes are extracted from Neveu's CSBP.

% is still an 
% the question of whether.... is still open

%generalize the REM-like trap model of ref. [14] (corresponding to a ?one-step? replica symmetry breaking (RSB) scheme [15]) to a fully foliated tree structure (?full replica symmetry breaking? [12]). A large amount of pa- pers already studied various types of dynamics on trees [19], directly inspired from Parisi?s ultrametric construction. 
%
%
%standard picture of the spin glass phase typically involves a highly complex landscape of the free energy function exhibiting many nested valleys organized according to some hierarchical tree structure
%
%describe the dynamics 
% We do expect that on longer time scales this breaks down.

%this REM-like behavior breaks down ``REM universality regime'' 
%type of behavior breaks down

%to describe dynamics on long times scales in terms of thermally activated barrier crossings.
%the aging regime present in the REM is also present in $p$-spin SK models in a restricted domain of the time-scales and temperatures
% more recently

%from which a detailed picture of the aging behavior of the REM 
%
%resulting in a detailed picture of the aging behavior of the REM 
%
%providing a detailed picture 

%The first rigorous connection between the microscopic dynamics of a spin system and a trap model was established in \cite{} in  \cite{} for the REM for a particular Glauber dynamics of the REM

%Main model that they intend to describe are 

%\subsection{Description of the model.} 
\subsection{Sasaki and Nemoto GREM-like trap model \cite{SN00a}}

\newcommand{\M}{\mathcal{M}}

We start by specifying the underlying tree structure. For $n\in\N$, we write $[n]=\{1,\dots,n\}$. Let $L\in\N$ be fixed and set $V|_k=[n]^k$ for $k=1,\dots,L$.  We define the rooted $L$-level perfect $n$-ary tree by
\begin{equation}
 \tre=\bigcup_{k=0}^L V|_k,
\end{equation}
where $V|_0=\{\emptyset\}$ is the root. We use the notation $\mu|_{k}=\mu_1\mu_2\cdots\mu_k$ for a generic element of $V|_k$. We sometimes simply use $\mu$ for $\mu|_L\in V|_L$. By convention the root belongs to the 0-th generation of the tree and $\mu|_k$ to the $k$-th generation. For $0\leq k_1<k_2\leq L$, we say that $\mu|_{k_1}\in V|_{k_1}$ is an ancestor of $\mu'|_{k_2}\in V|_{k_1}$ if $\mu|_{k_1}=\mu'|_{k_1}$.  The set $V|_L$
consists of the leaves of the tree, that is, the vertices that do not have any offspring. 
Whenever convenient we add the root to the notation by writing 
$\mu|_k=\mu_0\mu_1\cdots\mu_k$, where $\mu_0\equiv\emptyset$. Note that the trees $\tre$ are parametrized by $n\in\N$. However, for notational convenience we do not keep $n$ in the notation.
\begin{figure}[H]
	\begin{center}
		\includegraphics[trim = 25mm 20cm 0mm 25mm, clip, scale=0.9]{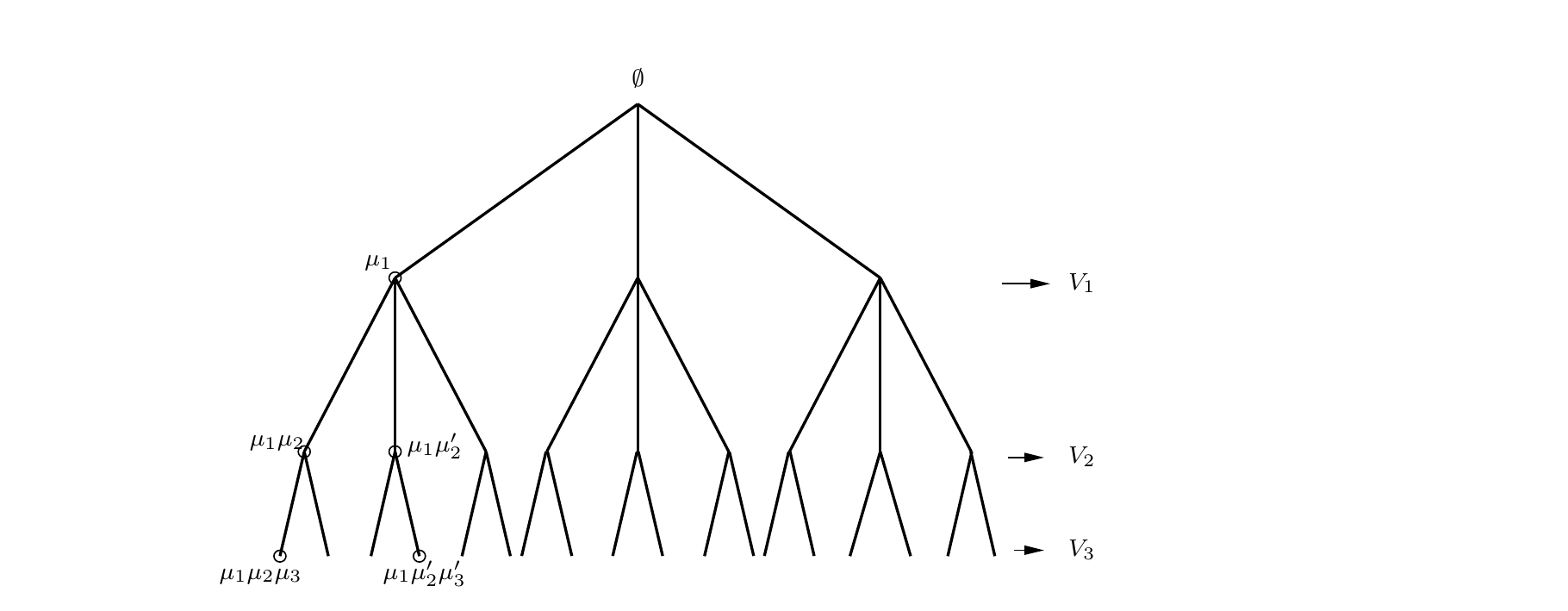}
	\end{center}
	\caption{A representation of $\tre$ with $L=n=3$.}
	\label{tree1}
\end{figure}
Given two vertices $\mu|_k,\mu'|_k\in V|_k$, we denote by 
\begin{equation}\label{glca}
 \big\langle\mu|_k,\mu'|_k\big\rangle = \max\bigl\{l: \mu|_l=\mu'|_l\bigr\},\quad \mu,\mu'\in V|_k,
\end{equation}
the generation of their last common ancestor (hereafter abbreviated g.l.c.a.).
The trapping landscape, or random environment, is a collection of independent random variables on the vertices of the tree $\tre$,
\newcommand{\kl}[1]{{{#1,L}}}
\begin{equation}\label{trapcoll}
\bigl\{\lambda(\mu|_k):\;\mu\in{V}|_L,\;k=1,\dots,L\bigr\},
\end{equation}
where $\lambda(\emptyset)\equiv 0$ and, for $k=1,\dots,L$, 
%the $\lambda^{-1}(\mu|_k)$'s are heavy tailed:
\begin{equation}\label{distlam}
\mathbb{P}\left(\lambda^{-1}(\mu|_k)\geq u\right)=u^{-\alpha_\kl{k}},\;\;\;u\geq 1.
\end{equation}
Here, the $\alpha_{k,L}$'s are real numbers satisfying
\begin{equation}\label{collalpha}
 0<\alpha_\kl{1}<\alpha_\kl{2}<\cdots<\alpha_\kl{L}<1,
\end{equation}
which implies that the $\lambda^{-1}$'s are heavy tailed.
We assume that the two-parameter sequences (in $n$ and $L$) of random environments are independent and 
defined on a common probability space $(\Omega,\mathcal{F},\mathbb{P})$. We denote by $\E$ the expectation  under $\mathbb{P}$.

For $k=1,\dots,L$, let $Y|_k$ be the discrete time Markov chain on $V|_k$ with transition probabilities 
\begin{equation}\label{transrates}
W_{k}(\mu|_k,\mu'|_k)=\sum_{l=0}^{\langle\mu|_k,\mu'|_k\rangle\wedge (k-1)} \frac
{\big(1-\lambda(\mu|_l)\big) \prod_{l'=l+1}^{k-1}\lambda(\mu|_{l'})}{n^{k-l}},
\end{equation}
with the convention that $\prod_{j=k}^{k-1}\lambda(\mu|_{j})=1$.
The GREM-like trap model, denoted by $(X_L(t):t\geq 0)$, is a continuous time Markov jump process on the set of leaves of the tree, $V|_L$, whose transition rates are given by
\begin{equation}\label{rates}
w_L(\mu,\mu')=\lambda(\mu)W_L(\mu,\mu').
\end{equation}
Thus, we see that after waiting at a leaf $\mu$ an exponential time with mean value $\lambda^{-1}(\mu)$, the process jumps to $\mu'$ with probability $W_L(\mu,\mu')$. We also see that $X_L$  is a reversible process whose unique invariant measure assigns to $\mu\in V|_L$ the mass $\prod_{k=1}^L\lambda^{-1}(\mu|_k)$.
%The measure that assigns to $\mu\in V|_L$ the mass $\prod_{k=1}^L\lambda^{-1}(\mu|_k)$ is the unique reversible invariant measure of $X_L$.
% indeed  $\sum_{\mu'}W_L(\mu,\mu')=1$.
%This model was first introduced by Sasaki and Nemoto in \cite{SN00a}. In the

A comment on the form of the transition probabilities  $W_L(\mu,\mu')$  is now in order. 
Note that $\lambda(\mu|_l)\in[0,1]$, $1\leq l\leq L$, so that the product
$
%p_l(\lambda(\mu))\equiv
(1-\lambda(\mu|_l))\prod_{l'=l+1}^{L-1}\lambda(\mu|_{l'})
$
appearing in the summation in (\ref{transrates}) is a probability. In physicists' term, this is the probability that, along the transition from $\mu$ to $\mu'$, the system is activated from $\mu$ to $\mu|_{l+1}$ but not from $\mu|_{l+1}$ to $\mu|_l$, meaning that the process jumps out of the traps attached to the vertices $\mu|_{l'}$, $l+1\leq l'\leq L-1$, but stays stranded in the trap attached to $\mu|_l$.
% in the trap attached to $\mu|_l$. 
The process then choses its next state uniformly at random among the $n^{L-l}$ leaves descending from $\mu|_l$.
% (This is uniform choice assumption is a main simplifying features of the trap model approach.)

Throughout this paper the initial distribution of $X_L$ is taken to be the uniform distribution on $V|_L$. 
We embed the distributions of the chains $X_L(t)$ for each $n\in\N$ and $L\in\N$, on a common probability space  whose distribution and expectation
we denote by $\mathcal{P}$ and $\mathcal{E}$, respectively. We suppress any references to the trapping landscape in the notation.

\iffalse
jumps the following way: first it draws a copy of the random variable $g_{\mu}$ (each time independent from everything else); then, if $g_{\mu}=k$, recalling that $\mu|_k$ denotes the unique ancestor of $\mu$ on the $k$-th level, $X_L$ 
moves uniformly at random to one of the descendants of $\mu|_k$ on $V_L$. Note that, the description we gave in the introduction of how the dynamics chooses its jumps through climbing through the tree follows from (\ref{gmu}). We say that a jump of the process from a leaf $\mu$ is on the $k$-th level if it is chosen from the descendants of the ancestor of $\mu$ on the $(k-1)$-th level (equivalently if $g_\mu=k-1$), and beyond and including the $k$-th level if it is a jump on $l$-th level with $l\leq k$. Note that, the mean waiting time of the dynamics belongs to the last level of the tree, and the effect of the environment of other levels is through the transition probabilities. More precisely, the number of trials at $\mu|_k$ until the dynamics is accepted to go up of has a geometric distribution with mean $\lambda^{-1}(\mu|_k)$. In terms of traps, $\lambda^{-1}(\mu|_k)$ denotes the depth of the trap at $\mu|_k$.

\fi 

The two-time correlation function that we use to quantify aging in the GREM-like trap model is defined as follows. For $k=1,\dots,L$, we first set  
\begin{equation}\label{decorfunc}
\Pi_k(t,s)=\mathcal{P}\left(\big\langle X_L(t),X_L(t+u)\big\rangle\geq k,\,\forall u \in [0,s]\right),\;\;t,s>0.
\end{equation}
We next choose a non-decreasing smooth function $q:[0,1]\to[0,1]$ with $q(0)=0$ and $q(1)=1$, and then use it to define the two-time correlation function
\begin{equation}\label{overlap}
C_L(t,s)=\sum_{k=1}^L \left[q\left(\frac{k}{L}\right)-q\left(\frac{k-1}{L}\right)\right] \Pi_k(t,s).
\end{equation}

Let us point out the key connections between the GREM-like trap model and the GREM (or the CREM, its generalization to continuous hierarchies).
%  -- focusing on the GREM for the sake of simplicity. 
 For this we rely on the paper \cite{BK07} by Bovier and Kurkova, that reviews the state of the art on the statics of these models. Using (\ref{glca}) one naturally defines a distance, $1-L^{-1}\big\langle\mu,\mu'\big\rangle$, $\mu,\mu'\in V|_L$, on the set of leaves of the tree, which is nothing but the ultrametric distance  (see (1.1) in \cite{BK07}) that endows the space of spin configurations of the GREM. 
Observe moreover that the function $q$ in  (\ref{overlap}) is the analogue of the function $A$ that enters the definition of
the covariance structure of the GREM (see (1.2) in \cite{BK07}).
Thus $C_L(t,s)$ in (\ref{decorfunc}) naturally plays the role of the spin-spin correlation function between two configurations of some microscopic GREM dynamics observed at times $t$ and $t+s$.
Let us now turn to the trapping landscape (\ref{trapcoll}). Its key features are modeled on
%inspired  from
those of the point process of extremes of the GREM's Botzman weights at  temperature $T$.
%, $e^{ H_n(\sigma)/T}$, where $H_n(\sigma)$ is the Hamiltonian and $T$ the temperature). 
Under certain conditions on the parameters of the model, this process is known to converge, as the system's size diverges, to  Ruelle's (non-normalized) Poisson cascades,
%(see Section 3 of \cite{BK07}), 
a hierarchical point process constructed from a collection of Poisson point processes of intensity $c_kx^{-(1+T/T_k)}dx$, one for each level, $k$, of the hierarchy; here $T_k$ is an associated critical temperature, and all processes that may result from this limiting procedure must be such that $T_1>\cdots>T_L$ (see Theorem 2.3, Theorem 3.2, and definition 3.1 in 
%or (2.8) 
in \cite{BK07}).
%The low temperature regime is then characterized by
Hence, at low enough temperature, $T/T_1<\cdots<T/T_L<1$.
%the $k$-th level process being a Poisson point process of intensity $c_kx^{-(1+T/T_k)}dx$, where $T_k$ is an associated critical temperature, and all processes that may result from this limiting procedure must be such that $T_1>\cdots>T_L$ (see Theorem 2.3, Theorem 3.2, and definition 3.1 in \cite{BK07}).
Setting $\alpha_{k,L}=T/T_k$ in (\ref{collalpha}), the non-normalized invariant measure $\prod_{k=1}^L\lambda^{-1}(\mu|_k) $, $\mu\in V|_L$, of the process $X_L$ in (\ref{rates}) models Ruelle's (non-normalized) cascade, while each $\lambda^{-1}(\mu|_k) $ models the depth of a trap at $\mu|_k$.
%the process must jump out in order to move to a leaf $\mu|_L$ that does not descend from $\mu|_k$.
Based on heuristic ideas derived from metastability, the GREM dynamics is then replaced in an ad hoc manner by the process $X_L$. We refer to Section 1.2 of \cite{BBG03b} for a more precise explanation of the correspondence between microscopic dynamics and trap models in the $1$-level (REM-like) trap model.
Note finally that in the GREM, the size of the hypercubes of different levels can vary and this correspond to varying branch sizes of the tree in the GREM-like trap model. Our results are valid for a range of trees with varying branch sizes, however, in order to keep the notations simple we opted to work with regular trees.

\begin{remark}
	Observe that the two-time correlation function $C_L$ can be expressed in terms of the distribution of the overlap observed by the dynamics. More precisely,
	\begin{equation}
	C_L(t,s)=\mathcal{E}\left(q\bigl(T_L(t,s)\bigr)\right),
	\end{equation}
	where
	\begin{equation}\label{zlts}
	T_L(t,s)=\sup\left\{l:\frac{1}{L}\langle X_L(t),X_L(t+u)\rangle\geq l,\,\forall u \in [0,s]\right\}.
	\end{equation}
Indeed, in our proofs we identify the limiting distribution of $(\ref{zlts})$, and thus, our results are valid for any choice of $q$.
\end{remark}
Following the by now well-established strategy to analyze the aging properties of disordered systems, for each level $k$, we introduce the so-called $k$-th level clock process, $S_{k,L}$, that is the partial sum process defined through
\begin{equation}
X_{L}(t)\big|_k=Y(S_{k,L}^{\leftarrow}(t))|_k,\quad t\geq 0.
\end{equation}
In contrast with earlier works, in order to fully describe the behavior of the two-time correlation function we need to control a whole collection of clock processes, one for each level of the tree.

The following description of $S_{k,L}$ is going to be useful. We say that a jump of $X_{L}$ is from a level $k\in\{1,\dots,L\}$, if the last activated vertex is on the $(k-1)$-th level of the tree, and that a jump is beyond and including the $k$th level if it is a jump from a level in $\{1,\dots,k\}$. Then, ${S}_{k,L}(i)$ is the time it takes for the dynamics $X_L$ to make $i$ jumps beyond and including the $k$-th level of the tree.

\iffalse
Note that, the event inside the probability term in (\ref{decorfunc}) is that the process does not jump beyond and including the $k$-th level of the tree on the whole time interval $[t,t+s]$. Another natural choice of event for (\ref{decorfunc}) is that the projection of the dynamics onto the $k$-th is at the same vertex at times $t$ and $t+s$. Our proofs for the aging levels are valid also with this choice of correlation function, with the same limiting aging functions. However, for the non-aging levels, although totally tractable, the results would be more complicated. Hence, we choose to work with the former, leading to more compact results.
\fi

\newcommand{\LL}{{(L)}}

\subsection{Convergence of the two-time correlation function.}
To study aging one has to choose a time scale $c_n$ on which the dynamics is observed. In this paper we work with $c_n$ of the form
\begin{equation}
 \tn=n^{\rho},\;\;\rho>0.
\end{equation}
We then say that the system is aging on the $k$-th level, for $\theta>0$, if $\Pi_k(\tn,\theta\tn)$ has a non-trivial limit that depends only on $\theta$, and that it is non-aging on the $k$-th level if $\Pi_k(\tn,\theta\tn)$ converges to 0. 
For $k=1,\dots,L$, we set
\qw\label{crtscales}
d_\kl{k}=\frac{1}{\alpha_\kl{L}}+\frac{1}{\alpha_\kl{L-1}}+\cdots \frac{1}{\alpha_\kl{k}}-(L-k).
\qwe
Note that by (\ref{collalpha}) 
\begin{equation}\label{ineqds}
 d_\kl{1}>d_{\kl{2}}>\cdots>d_\kl{L}>1.
\end{equation}
For a given exponent $\rho$, we define $\lag(\rho)\in\{1,\dots,L\}$ by
\begin{equation}\label{aglevels}
\lag(\rho)=\sup\{k:\rho< d_{k,L}\}.
\end{equation}  
For $\alpha\in(0,1)$ let $Asl_{\alpha}$ be the classical arcsine law distribution function
\begin{equation}
 Asl_{\alpha}(u)=\frac{\sin \alpha\pi}{\pi}\int_0^u x^{-\alpha}(1-x)^{\alpha-1}dx,\;\;u\in[0,1].
\end{equation}
\begin{theorem}\label{thrm1}
For any fixed $L\in\N$ and $\rho\in(0,d_{1,L})\setminus\{d_{L,L},\dots,d_{2,L}\}$ set $\lag=\lag(\rho)$ and $\bar{\alpha}_k=\prod_{i=k}^{\lag}\alpha_{i,L}$. There exists a subset $\Omega_L\subseteq \Omega$ with $\mathbb{P}(\Omega_L)=1$ such that for any environment in $\Omega_L$, for any $\theta>0$, 
\begin{equation}\label{corrlimit}
\lim_{n\to\infty}C_L\Big(\tn,\theta\tn\Big)=\sum_{k=1}^{\lag}\left[q\left(\frac{k}{L}\right)-q\left(\frac{k-1}{L}\right)\right] 
Asl_{\bar{\alpha}_{k}}\Big(\frac{1}{1+\theta}\Big).
\end{equation}
\end{theorem}
\begin{remark}
	Indeed, we are going to prove that, under the assumptions of Theorem \ref{thrm1}, for $k=1,\dots,l^*$, $\Pi_k(c_n,\theta c_n)\to Asl_{\bar{\alpha}_{k}}\Big(\frac{1}{1+\theta}\Big)$, and for $k=l^*(\rho)+1,\dots,L$, $\Pi_{k}(c_n,\theta c_n)\to 0$. 
\end{remark}
The collection in (\ref{crtscales}) can be seen as critical time scale exponents. More precisely, as a consequence of Theorem \ref{thrm1}, if $\rho<d_{k,L}$, the dynamics is aging on the $k$-th level, and if $\rho>d_{k,L}$, it is non-aging on the $k$-th level. Then, the levels $1,\dots,\lag(\rho)$ are aging whereas the levels below $\lag(\rho)$ are non-aging, and the inequalities in (\ref{ineqds}) reflect that as the time scale of observation gets longer the aging behavior disappears from bottom to the top of the tree.

Next, we take the $L\to\infty$ limit of the limiting two-time correlation function in (\ref{corrlimit}). In order to do this, we need to define a family of collections of $\alpha$ parameters satisfying (\ref{collalpha}). Let $R:[0,1]\to[0,\infty)$ be a strictly increasing, strictly concave, smooth function with $R(0)=0$. 
For $L\in\N$, set
\begin{equation}\label{Ralphas}
 \alpha_\kl{k}=\exp\Big\{-\Big(R\big(k/L\big)-R\big((k-1)/L\big)\Big)\Big\},\quad k=1,\dots,L.
\end{equation}
As a result, the collection $\{\alpha_{1,L},\dots,\alpha_{L,L}\}$ satisfies (\ref{collalpha}) for all $L\in\N$. Note that although for a given collection $\{\alpha_{1,L},\dots,\alpha_{L,L}\}$ as in (\ref{collalpha}), one can find a function $R$ satisfying (\ref{Ralphas}), it is not necessarily possible to do that for a given family of such collections $\{\alpha_{1,L},\dots,\alpha_{L,L}:\;L\in\N\}$. Also observe that, in the light of Theorem \ref{thrm1}, the choice of (\ref{Ralphas}) is natural for an infinite levels limit because the infinite product of $\alpha$ parameters is non-trivial.

For a given exponent $\rho>0$, we set
\begin{equation}\label{contla}
r^*(\rho)=\sup\{s\geq 0: R(1)-R(s)+1>\rho\}.
\end{equation}
We will see in Section \ref{Section3} (see (\ref{sonispat2}) and the paragraph before it) that $r^*(\rho)$ is the limit of the ratio $\lag(\rho)/L$ as $L$ diverges. As a result, there is aging for $\rho$ such that $r^*(\rho)>0$. Since $R(0)=0$, the latter is equivalent to $\rho<R(1)+1$.
\begin{corollary}\label{corol1}
Let $\rho\in(0,R(1)+1)\setminus \{d_{L,L},\dots,d_{2,L}\}$ for all $L$ large enough. Set $r^*=r^*(\rho)$ and let 
$\alpha(x)=\exp\big\{-(R(r^*)-R(x))\big\}$ for $x\in[0,r^*]$. There exists a subset $\Omega'\subseteq \Omega$ with $\mathbb{P}(\Omega')=1$ such that for any environment in $\Omega'$, for any $\theta>0$,
\begin{equation}\label{eqcorol1}
\lim_{L\to\infty}\lim_{n\to\infty}C_L\Big(\tn,\theta\tn\Big)=\int_0^{r^*}q'(x) Asl_{\alpha(x)}\Big(\frac{1}{1+\theta}\Big)dx.
\end{equation}
\end{corollary}
\begin{remark}
	The result in Corollary \ref{contla} can be extended to functions $q$ that are right continuous and left limits. However, since the formula in \ref{eqcorol1} is very transparent and neat, we use only smooth functions.
\end{remark}
Let us discuss the connection of our results with earlier literature and comment about some directions for future research. Observe that our results do not cover the time scales $c_n\sim n^{d_{k,L}}$. Moreover, varying branch sizes can yield time scales that are critical for several levels simultaneously. For these levels the inequalities in (\ref{ineqds}) become equalities. The cases where the branch sizes are fine tuned in such a way that all the critical time scale exponents are the same were studied in \cite{FGG12}. There, the authors constructed a $K$ process in an infinite $L$-level tree and proved that for any critical time scale, the scaling limit of the GREM-like trap model exists and is given by such an $L$-level $K$ process. In \cite{FP14}, the authors constructed, for a special choice of parameters, an infinite levels, infinite volume $K$ process and showed that it is the limit $L\to\infty$ of the $L$-level $K$ process obtained in \cite{FGG12}. An investigation of intermediate cases, where there are times scales that are critical for some levels and aging for others, would require a non-trivial combination of analysis of this paper and that of \cite{FGG12}. 

\subsection{Convergence of the clock processes.}
In this section we state our results on the convergence of the clock processes. We express the limiting clock processes using Neveu's CSBP which is a time-homogeneous Markov process $(W(r):r\geq 0)$ whose semigroup is characterized by
\begin{equation}\label{lapW}
E\left[e^{-\kappa W(r)}|W(0)=t\right]=\exp\big({-t \kappa^{e^{-r}}}\big),\quad\kappa>0,t\geq 0.
\end{equation}
We write $W(r,t)$ for $W$ starting from $t\geq 0$, that is, $W(0,t)=t$. Using Kolmogorov's extension theorem one can construct a process $(W(r,t):r,t\geq 0)$ such that $W(\cdot,0)\equiv 0$ and, $W(\cdot,t+s)-W(\cdot,t)$ is independent of $(W(\cdot,c):0\leq c\leq t)$ and has the same law as $W(\cdot,s)$. Hence, for any fix $r\geq 0$, the right continuous version of $W(r,\cdot)$ has independent, stationary increments.\,From its Laplace transform we see that it is an $e^{-r}$-stable subordinator with Laplace exponent $\kappa^{e^{-r}}$. 
Moreover, by (\ref{lapW}), $W(r+p,\cdot)$ has the same distribution as the Bochner subordination of $W(r,\cdot)$ with the directing process
$W'(p,\cdot)$, where $W'(p,\cdot)$ is an independent copy of $W(p,\cdot)$, that is,
\begin{equation}\label{bochner}
W(r+p,\cdot)\overset{d}=W(r,W'(p,\cdot)).
\end{equation}    
The above description is taken from \cite{BL00} where it is pointed out that, in general, (\ref{bochner}) allows to connect CSBPs and Bochner's subordination, a connection developed further in the following:
\begin{proposition}[Proposition 1 in \cite{BL00} applied to Neveu's CSBP]\label{bohsub} On some probability space there exists a process $\big(Z_{p,r}(t):0\leq p\leq r,\text{ and }t\geq 0\big)$ such that:
\item{(i)} For every $0\leq p\leq r$, $Z_{p,r}=\big(Z_{p,r}(t):t\geq 0\big)$ is an $e^{-(r-p)}$-stable subordinator with Laplace exponent $\kappa^{e^{-(r-p)}}$.
\item{(ii)} For any integer $m\geq 2$ and $0\leq r_1\leq \cdots\leq r_m$ the subordinators $Z_{r_1,r_2},Z_{r_2,r_3},\dots,\\Z_{r_{m-1},r_{m}}$ are independent and
\begin{equation}\label{bohsubii}
Z_{r_1,r_m}(t)=Z_{r_{m-1},r_m}\circ\cdots\circ Z_{r_{1},r_2}(t),\;\;\;\forall t\geq 0 \;\;a.s.
\end{equation}
Finally, $\big\{Z_{0,r}(t):r\geq 0,t\geq 0\big\}$ and $\big\{W(r,t):r,t\geq 0\big\}$ have the same finite dimensional distributions.
\end{proposition}

To each level of the tree we assign a pair of sequences, $\aan{k}$ and $\ccn{k}$, defined by
\begin{equation}\label{scalrel}
\aan{k}=\left\{\begin{array}{ll}n^{1+\rho-d_{k,L}}&k=L,\dots,\lag+1,\\  n^{\bar{\alpha}_{k}(1+\rho-d_{\lag,L})}&k=\lag,\dots,1,\end{array}\right.
\end{equation}
and
\begin{equation}
\ccn{k}=\aan{k+1},\quad k=1,\dots,L-1, \quad \ccn{L}=\tn.
\end{equation}
Note that as $n\to\infty$, $a_n(k)\ll n$ and $\ccn{k}=\aan{k}^{1/\alpha_{k,L}}$ for all $k\leq \lag$ (i.e. for all the aging levels), whereas $\aan{k}\gg n$ and $\ccn{k}=\aan{k}n^{1/\alpha_{k,L}-1}$ for all $k\geq \lag+1$ (i.e. for all the non-aging levels). Moreover, under the assumptions of Theorem \ref{thrm1}, in both cases the decay or growth of $a_n(k)/n$ is at least polynomial. 

For $k=1,\dots,L$, we define the rescaled clock processes 
\begin{equation}\label{scalclock}
S^\nn_{k,L}(t)=\frac{S_{k,L}(\lfloor t \aan{k}\rfloor)}{\tn},\quad t\geq 0,
\end{equation}
where we set $S_{k,L}(0)=0$. Hence, $S^\nn_{k,L}\in D([0,\infty))$ where $D([0,\infty))$ denotes the space of c\`adl\`ag functions on $[0,\infty)$. The following is our main result on the convergence of clock processes. 

\begin{theorem}\label{fixL}
For any $L\in\N$ there exists a subset $\widetilde{\Omega}_L\subseteq\Omega$ with $\mathbb{P}(\wt\Omega_L)=1$ such that for some positive constants $b_{1,L},\dots,b_{\lag,L}$ setting
\begin{equation}
\widetilde{Z}_{k,\lag}(\cdot)=Z_{R((k-1)/L),R(\lag/L)}(b_{k,L}\,\cdot),
\end{equation}
for any environment in $\wt\Omega_L$, as $n\to\infty$
\begin{equation}\label{eqfixL}
\Big(\sll{k}{L}^\nn:\;k=1,\dots,\lag\Big)\Longrightarrow \Big(\widetilde{Z}_{k,\lag}:\;k=1,\dots,\lag\Big)
\end{equation}
weakly on the space $D^{\lag}([0,\infty))$ equipped with the product Skorohod $J_1$ topology.
\end{theorem}

\begin{remark}
Note that by (\ref{Ralphas}) and the definition of Neveu's CSBP, for any $i=1,\dots,L$ and $b>0$, $Z_{R((i-1)/L),R(i/L)}(b\,\cdot)$ is a stable subordinator with index $\alpha_{i,L}$. Therefore, by (\ref{bohsubii}) the distribution of the right hand side of (\ref{eqfixL}) is given by compositions of stable subordinators.
\end{remark}

\begin{remark}
Together with the previous remark, Theorem \ref{fixL} implies that $S_{k,L}^\nn$ converges weakly to an $\bar{\alpha}_k$-stable subordinator, where $\bar{\alpha}_k$ is given as in Theorem \ref{thrm1}. If one is only interested in marginal distributions of the clock processes (which is enough to obtain our results on two-time correlation functions) a shorter proof is available through a technique based on Durrett and Resnick \cite{DR78}, which has been recently proved to be very useful in the context of dynamics in disordered systems,~see \cite{BG13,BGS13,Ver1,Ver2,G12}. However, in this paper we choose to prove the stronger result of the joint convergence of clock processes in order to make the connection to Neveu's CSBP more transparent. 
\end{remark}

In \cite{BL00}, Bertoin and Le Gall gave a representation of the genealogical structure of CSBPs using Bochner's subordination. We say that an individual $t$ at generation  $r$ has an ancestor $c$ at generation $p\in[0,r]$ if $c$ is a jumping time of $Z_{p,r}$ and
\begin{equation}
Z_{p,r}(c-)<t<Z_{p,r}(c).
\end{equation}
Since the L\'evy measure of $Z_{r,p}$ has no atoms, the set of individuals at generation $d$ who do not have an ancestor at generation $r<d$ has a.s. Lebesgue measure 0. In view of this, for individuals $t_1$ and $t_2$ at generation $r$, we let
\begin{equation*}
T_{r}(t_1,t_2)=\sup\{p\geq 0: t_1 \text{ and } t_2 \text{ have a common ancestor at generation } p  \},
\end{equation*}
and set $T_{r}(t_1,t_2)=-\infty$ if $t_1$ and $t_2$ do not have a common ancestor. Using $T_r$ we can express the limiting two-time functions in Theorem \ref{thrm1} and Corollary \ref{corol1} as follows. Note that, by definition, we have
\begin{equation}\label{ptrr}
 P(T_r(t_1,t_2)\geq p)=P(\{Z_{p,r}(t):t\geq 0\}\cap [t_1,t_2]=\emptyset),
\end{equation}
and since $Z_{p,r}$ is a stable subordinator with index $e^{-(r-p)}$, the right hand side of (\ref{ptrr}) is nothing but $Asl_{\alpha}(b/d)$, where $\alpha=e^{-(r-p)}$. Thus, we can rewrite the $Asl$ terms in Theorem \ref{thrm1} and Corollary \ref{corol1} as, respectively,
\begin{equation}
\quad \quad  \;\;\;\;\;\;\;\;Asl_{\bar{\alpha}_k}\Big(\frac{1}{1+\theta}\Big)=P\Big(T_{R(\lag/L)}(1,1+\theta)\geq R((k-1)/L)\Big),
\end{equation}
and
\begin{equation}
 Asl_{\alpha(x)}\Big(\frac{1}{1+\theta}\Big)=P\Big(T_{R(r^*)}(1,1+\theta)\geq R(x)\Big).
\end{equation}

\begin{remark}
Neveu's CSBP was first used in the study of the statics of the GREM and CREM \cite{Nev92}. Namely, the limiting geometric structure of the Gibbs measure of these models can be expressed in terms of the genealogy of Neveu's CSBP (see Section 5 of \cite{BK07}). However, in the context of the dynamics, it appears in a different way, describing the limits of the clock processes.
\end{remark}

The rest of this paper is organized as follows. In Section \ref{section2} we describe clock processes through a certain cascade of point processes associated to the dynamics, prove that they converge weakly to a cascade of Poisson point processes, and finally, using little more than the continuous mapping theorem, we establish Theorem \ref{fixL}. In Section \ref{Section3} we prove Theorem \ref{thrm1} and Corollary \ref{corol1}.

\section{Convergence of the clock processes.}\label{section2}

\subsection{Description of the clock processes through a cascade of point processes.}

We first give definitions and introduce notation for general cascade processes.

\newcommand{\vep}{\varepsilon}

For a complete, separable metric space $A$, we designate by $M(A)$ the space of point measures on $A$, and by $\vep_x$ the Dirac measure at $x\in A$, i.e. $\vep_x(F)=1$ if $x\in F$ and $\vep_x(F)=0$ if $x\notin F$. We set
\begin{equation} 
H=(0,\infty)\times (0,\infty), \quad \text{and}\quad M_{1,l}=M(H^1)\otimes \cdots \otimes M(H^{l}),\;\; l\in\N.
\end{equation}
All the point measures we use are indexed by $\N^k$ and we use the notation $j|_k=j_1j_2\cdots j_k$ for a member of $\N^k$. 

The set of $l$-level cascade point measures, $\mathcal{M}_{1,l}$, is the subset of $M_{1,l}$ where for each $m=(m_1,\dots,m_l)\in\M_{1,l}$ there corresponds a collection of points in $H$ of the form
\begin{equation}\label{collson}
\left\{(t_{j|_k},x_{j|_k}):j|_l\in\N^l,k=1,\dots,l\right\}
\end{equation}
such that for each $k=1,\dots,l$
\begin{equation}\label{mkcoll}
m_k=\sum_{j|_k\in\N^k}\varepsilon_{(t_{j_1},x_{j_1},\dots,t_{j|_k},x_{j|_k})}.
\end{equation}
\iffalse
\begin{equation}
m_k(\cdot)=\sum_{j_1,\dots,j_k\in\N}\mathds{1}\left\{\bigl(t_{j_1},x_{j_1},\dots,t_{j|_k},x_{j|_k}\bigr)\in\;\cdot\right\}.
\end{equation}
\fi
We refer to the collection in (\ref{collson}) as the marks of $m$. Throughout this paper we assume that all the point measures are simple.  Let $\widetilde{\M}_{1,l}$ be the subset of $\M_{1,l}$ such that for each $m\in\widetilde{\M}_{1,l}$, $k=1,\dots,l$, $j|_{k-1}\in\N^{k-1}$ and $t> 0$,
\begin{equation}
m_k\Bigl(\{(t_{j_1},x_{j_1},\dots,t_{j|_{k-1}},x_{j|_{k-1}})\}\times \bigl((0,t]\times(0,\infty)\bigr)\Bigr)<\infty.
\end{equation}
Then, for $m\in\wt{\M}_{1,l}$, there exists a unique labeling of the marks so that for any $j|_{k-1}\in\N^{k-1}$
\begin{equation}\label{ordered}
t_{j|_{k-1}1}<t_{j|_{k-1}2}<\cdots. 
\end{equation}
From now on we only use this labeling for the marks of $m\in\wt{\M}_{1,l}$.

We define $T_l:\wt{\M}_{1,l}\to D([0,\infty))$ by
\begin{equation}
T_l(m)(t)=\sum_{t_{j_1}\leq t}\;\sum_{\;\;t_{j_1j_2}\leq x_{j_1}}\cdots \sum_{\;\;\;\;\;\;t_{j|_{l}}\leq x_{j|_{l-1}}}x_{j|_l}.
\end{equation}
\newcommand{\ovl}{\overline}
Next, we introduce a map $\ovl{T}_l:\wt{\M}_{1,l}\to\wt{\M}_{1,l-1}$. For $m\in\wt{\M}_{1,l}$ let
\begin{equation}
Z(i)=m_2\Bigl(\{(t_{i},x_{i})\}\times \bigl((0,x_{i}]\times(0,\infty)\bigr)\Bigr), \quad i\in\N,
\end{equation}
where we set $Z(0)=0$, 
\begin{equation}
g(i)=\max\{r:Z(0)+\cdots+Z(r)< i\},\;h(i)=Z(0)+\cdots +Z(g(i)), \quad i\in\N,
\end{equation}
and
\begin{equation}
s(i)=x_1+\cdots +x_{i},\quad i\in\N,\quad s(0)=0.
\end{equation}
Then, we define $\overline{T}_l(m)$ as a point in $\M_{1,l}$ whose marks are given by
\begin{equation}\label{tmarks}
\left\{(\overline{t}_{ j|_k},\overline{x}_{j|_k}):j|_{l-1}\in \N^{l-1},k=1,\dots,l-1\right\},
\end{equation}
where, for $j_1\in\N$,
\begin{equation}
\overline{t}_{j_1}=s(g(j_1))+t_{(g(j_1)+1)(j_1-h(j_1))},\;\overline{x}_{j_1}=x_{(g(j_1)+1)(j_1-h(j_1))},
\end{equation}
and for $k>1$ and $j|_k\in\N^{k}$,
\begin{equation}
\overline{t}_{j|_k}=t_{(g(j_1)+1)(j_1-h(j_1))j_2\dots j_k},\;\overline{x}_{j|_k}=x_{(g(j_1)+1)(j_1-h(j_1))j_2\dots j_k}.
\end{equation}
It is clear that $\ovl{T}_l(m)\in \wt{\M}_{1,l}$ and the marks in (\ref{tmarks}) is already ordered in jump times, that is, (\ref{ordered}) is satisfied.

\newcommand{\xtl}{\tilde{{x}}}

For $k\leq l$, we define $T_{k,l}: \wt{\M}_{1,l}\to D([0,\infty))$ by
\begin{equation}\label{impcomp}
T_{k,l}=T_{l-k+1}(\ovl{T}_{l-k+2}\circ\cdots\circ \ovl{T}_{l-1}\circ \ovl{T}_l),
\end{equation}
where it is understood that $T_{1,l}=T_l$. 

For $l_1\leq l_2$ and $m=(m_1,\dots,m_{l_2})\in\wt{\M}_{1,l_2}$, let $m|_{l_1}=(m_1,\dots,m_{l_1})\in\wt{\M}_{1,l_1}$. 
%For $k\leq l_1$ we can extend the domain of $T_{k,l_1}$ to $\wt{\M}_{1,l_2}$ by setting for $m\in\wt{\M}_{1,l_2}$, $T_{k,l_1}(m)=T_{k,l_1}(m|_{l_1})$. 
We use the following property later: for $k<l_1\leq l_2$ and $m\in\wt{\M}_{1,l_2}$,
\begin{equation}\label{identity}
T_{k,l_2}(m)=T_{l_1,l_2}(m)\circ T_{k,l_1-1}(m|_{l_1-1}),
\end{equation}
where the composition in the above display is on the space $D([0,\infty))$.

We now describe a cascade of point processes associated to the dynamics $X_L$ whose image under the functionals $T_{k,L}$ yields the clock processes. A cascade of simple random walks with $L$ levels on $\tre$ is a collection of random variables
\begin{equation}\label{colldjc}
J=\Big\{J_k(j_k;j|_{k-1}):\;j|_L\in\N^L,\;k=1,\dots,L\Big\} 
\end{equation}
that is characterized as follows:

\noindent{(i)} For each $k=1,\dots,L$ and $j|_{k-1}$ fixed, $\{J_k(j_k;j|_{k-1}):j_k\in\N\}$ 
is a collection of i.i.d. random variables distributed uniformly on $[n]$.

\noindent{(ii)} The families $\big\{J_k(j_k;j|_{k-1}):j_k\in\N\big\}$ for $k=1,\dots,L$ and $j|_{k-1}$ are independent.

For $k=1,\dots,L$, the jump chain $J(j|_k)$ on $V|_k$ is defined by
\begin{equation}
 J(j|_k)=J_1(j_1)J_2(j_2;j|_1)\cdots J_k(j_k;j|_{k-1}).
\end{equation}
In (i) above we make use of the fact that a simple random walk on the complete graph $[n]$ starting from a uniform distribution is the same as a sequence of i.i.d. uniform distributions on $[n]$. For a fixed realization of the random environment, consider the collection of random variables
\begin{equation}\label{colljumps}
\left\{(t_{j|_k},\xi_{j|_k}):j|_L\in\N^L,k=1,\dots,L\right\}
\end{equation}
given as follows: for a given realization of $J$, it is an independent collection with $\xi_{j|_k}\overset{d}=G(\lambda(J(j|_k)))$ for $k=1,\dots,L-1$, and $\xi_{j|_L}\overset{d}=\lambda^{-1}(J(j|_L))e$. Here, $G(p)$ denotes a geometric random variable with success probability $p$ and $e$ is a mean one exponential random variable. We also set $t_{j|_k}=j_k$. Using the collection in (\ref{colljumps}) we define $\zeta^L=(\zeta^L_{1},\dots,\zeta^{L}_{L})$ by
\begin{equation}
\zeta_{k}^L=\sum_{j_1,\dots,j_k\in\N}\vep_{(t_{j_1},\xi_{j_1},\dots,t_{j|_k},\xi_{j|_k})}.
\end{equation}
Clearly $\zeta^L\in\widetilde{\M}_{1,L}$ and the marks of it are already ordered in jump times. Finally, the clock process $S_{k,L}$ is given by
\begin{equation}
S_{k,L}=T_{k,L}(\zeta^L),\quad k=1,\dots,L.
\end{equation}

\iffalse
We can describe the rescaled clock processes as functions of rescaled cascade process of the dynamics. More precisely, defining $\zeta^{\nn,L}=(\zeta_{1,L}^{\nn,L},\dots,\zeta_{L,L}^{\nn,L})$ by
\begin{equation}
\zeta^{\nn,L}_{k,L}=\sum_{j_1,\dots,j_k}\vep_{(t_{j_1}/\aan{1},\eta_{j_1}/\ccn{1},\dots,t_{j|_k}/\aan{k},\eta_{j|_k}/\ccn{k})},
\end{equation}
we get
\begin{equation}
S_{k,L}^\nn=T_{k,L}(\zeta^{\nn,L}).  
\end{equation}
\fi

\subsection{Convergence of the cascade of point processes of the dynamics}

\newcommand{\X}{\mathit{X}}
\newcommand{\x}[1]{x^\nn_{j_1\dots j_{#1}}}
\newcommand{\xsq}[2]{\x{#1}:#1=1,\dots,#2;\;j_1,\dots j_{#1}\in\N}
\newcommand{\F}{{\mathcal{F}}}
\newcommand{\bS}{{S}}

We have seen in the previous section how the collection in (\ref{colljumps}) is used to describe all the clock processes. Since we are interested in proving the joint convergence of clock processes $S_{1,L},\dots,S_{\lag,L}$ rescaled as in (\ref{scalclock}), where $l^*=l^*(\rho)$, we consider the following collection 
\begin{equation}
\Big\{(t_{j|_k}^\nn,\xi_{j|_k}^\nn):\; j|_\lag\in\N^{\lag},k=1,\dots,\lag\Big\}
\end{equation}
obtained from (\ref{colljumps}) by setting $ t_{j|_k}^\nn={j_k}/{\aan{k}}$, 
\begin{equation}
\xi_{j|_k}^\nn=\frac{\xi_{j|_k}}{\ccn{k}}, \quad k=1,\dots,\lag-1,\;\text{ and } \;\xi_{j|_{\lag}}^\nn =\frac{\Lambda(j|_{\lag})}{\tn},
\end{equation}
where
\begin{equation}\label{lamvar}
 \Lambda(j|_{\lag})=\sum_{j_{\lag+1}=1}^{\xi_{j|_{\lag}}} \cdots \sum_{j_{L}=1}^{\xi_{j|_{L-1}}}\xi_{j|_L}.
\end{equation}
Here, the terms corresponding to the aging level $\lag$ collects the whole waiting times until the dynamics jump over the vertex $J(j|_{\lag})$. Defining $\zeta^{\nn,\lag}=(\zeta^{\nn,\lag}_{1},\dots,\zeta^{\nn,\lag}_{\lag})$ by
\begin{equation}\label{rcoll}
 \zeta_{k}^{\nn,\lag}=\sum_{j|_k\in\N^k}\vep_{(t_{j_1}^\nn,\xi_{j_1}^\nn,\dots,t_{j|_k}^\nn,\xi_{j|_k}^\nn)},
\end{equation}
we get 
\begin{equation}
S_{k,L}^\nn=T_{k,\lag}(\zeta^{\nn,\lag}),\quad k=1,\dots,\lag.
\end{equation}
Now we describe a cascade of Poisson point processes (PPP). For $l\in\N$, for constants $0<\beta_1<\cdots<\beta_l<1$ and $D_1,\dots, D_l>0$, let 
\begin{equation}
 \Big\{(t_{j|_k},\eta_{j|_k}):\; j|_k\in \N^{l},\;k=1,\dots,l\Big\}
\end{equation}
be a collection of random variables whose distribution is characterized by the following properties:

\noindent{(i)} For each $k=1,\dots,l$ and $j|_{k-1}$ fixed, the distribution of $\{(t_{j|_k},\eta_{j|_{k-1}j_k}):j_k\in\N\}$ 
is that of marks of a PPP on $H$ with mean measure $dt\times D_k\beta_k\; x^{-1-\beta_k}dx$.

\noindent{(ii)} The families $\big\{(t_{j|_{k-1}j_k},\eta_{j|_{k-1}j_k}):j_k\in\N\big\}$ for $k=1,\dots,l$ and $j|_{k-1}$ are independent.

We define $\chi^{l}=(\chi_{1}^{l},\dots,\chi_{l}^{l})$ by
\begin{equation}\label{poscas}
 \chi_{k}^{l}=\sum_{j|_k\in\N^k} \vep_{(t_{j_1},\eta_{j_1},\dots,t_{j|_k},\eta_{j|_k})},
\end{equation}
\iffalse
\begin{equation}\label{poscas}
 \zeta_{k,l}(\cdot)=\sum_{j_1,\dots,j_k\in\N}\mathds{1}\Big\{(t_{j_1},\xi_{j_1},\dots,t_{j_1\dots j_k},\xi_{j_1\dots j_k})\in\cdot\Big\},
\end{equation}
\fi
and refer to $\chi^l$ as Ruelle's Poisson Cascade (RPC) with parameters $\beta_1,\dots,\beta_l$ and constants $D_1,\dots,D_l$.

Our first goal is to prove that cascade of point measures $\zeta^{\nn,\lag}$  converges weakly to an RPC. Next, we prove that the versions of clock process functionals where the very small jumps are ignored are continuous. Finally, using the continuous mapping theorem and controlling the very small jumps we establish the convergence of clock processes. In this regard, we extend the classical results on convergence of sum of random variables with heavy tails (see e.g. Theorem 3.7.2 in \cite{DurrettBook} and Proposition 3.4. in \cite{ResNotes}).

Next, we state our weak convergence result. Recall that $\zeta^{\nn,\lag}$ and $\chi^{\lag}$ are random elements of product of space of point measures $M_{1,\lag}=M(H)\otimes \cdots M(H^\lag)$, and we use the weak convergence induced by the product vague topology, denoted by $\Rightarrow$.

\newcommand{\xinl}{\zeta^{\nn,\lag}}
\newcommand{\xil}{\chi^{\lag}}

\begin{proposition}\label{weaklemma}
There exist positive constants $D_1,\dots,D_{\lag}$ and a subset $\Omega'_L\subseteq \Omega$ with $\mathbb{P}(\Omega'_L)=1$ such that for any environment in $\Omega'_L$, as $n\to\infty$
\begin{equation}
\zeta^{\nn,\lag} \Rightarrow \chi^{\lag}, %Same thing here with \big vs. \left
\end{equation}
where $\chi^{\lag}$ is a RPC with parameters $\alpha_{1,L},\dots,\alpha_{\lag,L}$ and constants $D_1,\dots,D_\lag$.
\end{proposition}

We first recall some basic facts about Laplace functionals of point processes since we use them to prove Proposition \ref{weaklemma}. For the basic concepts about point processes we mainly follow the book \cite{Re87} and refer readers to the same source for further details. 

Let $m\in M(A)$ and $f$ be a non-negative Borel measurable function on $A$. Define
\begin{equation}
 m(f)=\int_A f(x)m(dx).
\end{equation}
Let $C_K^+(A)$ be the set of non-negative continuous functions on $A$ with compact support. Then, $m^\nn\in M(A)$ converges vaguely to $m\in M(A)$ if
\begin{equation}
 m^\nn(f)\to m(f),\quad \forall f\in C_K^+(A).
\end{equation}
The vague topology on $M(A)$ induced by the vague convergence is metrizable as a complete separable metric space. The weak convergence in $M(A)$ is with respect to the vague topology on $M(A)$.

Let $N$ be a point process on $A$ and let $P$ and $E$ denote its distribution and expectation, respectively. The Laplace functional of $N$ is a map which takes Borel measurable non-negative functions into $[0,\infty)$, defined by 
\begin{equation}
 \Phi_N(f)=E\left[\exp(-N(f))\right]=\int_{M(A)}\exp\left(-m(f)\right)P(dm).
\end{equation}
The Laplace functionals give a useful criteria for the weak convergence of point processes: $N^\nn\Rightarrow N$ if and only if $\Phi_{N^\nn}(f)\to\Phi_N(f)$ for all $f\in C_K^+(A)$. 

In Proposition \ref{weaklemma} we are concerned with sequences of vectors of point processes so now we describe the Laplace functionals for such vectors. Let $l\in\N$ and $N=(N_1,\dots,N_{l})$ be a random variable on ${M}_{1,l}$ and again $P$ and $E$ denote its distribution and expectation, respectively. The Laplace functional of $N$ is a map which takes a vector of $l$ non-negative Borel measurable functions, $f_1,\dots,f_{l}$ on $H,\dots,H^{l}$, respectively, into $[0,\infty)$, defined by
\begin{equation}
\begin{aligned}
 \Phi_{N}(f_1,\dots,f_{l})&=E\left[\exp\left(-f_1(N_1)-\cdots -f_{l}(N_{l})\right)\right]\\&=\int_{{M}_{1,l}}\exp\left(-m_1(f_1)-\cdots-m_{l}(f_{l})\right)P(dm_1\dots dm_{l}).
\end{aligned}
\end{equation}
As already stated in Proposition \ref{weaklemma}, the weak convergence we use is the one induced by the product vague topology on ${M}_{1,l}$.
{It is trivial to extend the characterization of the weak convergence of point processes by the Laplace functionals to the product space to get}
\begin{equation}
 N^\nn=(N_1^\nn,\dots,N^\nn_l)\Rightarrow N=(N_1,\dots,N_l),
\end{equation}
if and only if for all $f_1\in C_K^+(H)$,$\dots,$ $f_l\in C_K^+(H^l)$,
\begin{equation}
\Phi_{N^\nn}(f_1,\dots,f_l)\to \Phi_{N}(f_1,\dots,f_l).
\end{equation}
Thus, our goal is to show that $\Phi_{\xinl}\to\Phi_{\xil}$.

For $k=1,\dots,\lag-1$ and $\mu|_k\in V|_k$, let $\Pp^\nn_{\mu|_k}$ be the probability distribution
\begin{equation}\label{gdist}
\begin{aligned}
\Pp^\nn_{\mu|_k}\bigl((u,\infty)\bigr)&=\Pp\Big(G(\lambda(J(j|_k)))\geq u \ccn{k} \big|J(j|_k)=\mu|_k\Big),\;u>0,\\&=\Pp\Bigl(G(\lambda(\mu|_k))\geq u\ccn{k}\Bigr),
\end{aligned}
\end{equation}
and for $\mu|_{\lag}\in V|_{\lag}$, let $\Pp^\nn_{\mu|_{\lag}}$  be the probability distribution
\begin{equation}\label{pmuldef}
\Pp^\nn_{\mu|_{\lag}}\bigl((u,\infty)\bigr)=\Pp\Big(\Lambda(j|_{\lag})\geq u \tn \big|J(j|_{\lag})=\mu|_{\lag}\Big),\;u>0,
\end{equation}
where $\Lambda$ is given by (\ref{lamvar}).
We set
\begin{equation}
\bar{\nu}_k^\nn(u;\mu|_{k-1})=\frac{\aan{k}}{\mmn{k}}\sum_{\mu_k=1}^{\mmn{k}} \Pp_{\mu|_k}^\nn\bigl((u,\infty)\bigr),\; u>0.
\end{equation}
For simplicity, in the rest of this section we write $\alpha_1,\dots,\alpha_L$ for $\alpha_{1,L},\dots,\alpha_{L,L}$. The following lemma is the basic step in proving Proposition \ref{weaklemma}.
\begin{lemma}\label{lemanara}
For each $k=1,\dots,\lag$, there exist a subset $\Omega_{k,L}\subseteq \Omega$ with $\mathbb{P}(\Omega_{k,L})=1$ such that the following holds:
\;\item[(i)] for $k=1,\dots,\lag-1$, for any environment in $\Omega_{k,L}$, for any $u>0$, uniformly in $\mu|_{k-1}$, as $n\to\infty$
\begin{equation}\label{eqnlemma1}
\bar{\nu}_k^\nn(u;\mu|_{k-1})\rightarrow D_k u^{-\alpha_k},
\end{equation}
where $D_k=\Gamma(1+\alpha_k)$,
\;\item[(ii)] for $k=\lag$, for any environment in $\Omega_{\lag,L}$, for any $u>0$, uniformly in $\mu|_{\lag-1}$, as $n\to\infty$
\begin{equation}\label{eqnlemma2}
\bar{\nu}_{\lag}^\nn(u;\mu|_{\lag-1})\rightarrow D_{\lag}u^{-\alpha_{\lag}},
\end{equation}
where $D_{\lag}$ is positive a deterministic constant.
\iffalse
Moreover, let $Z$ be a positive random variables with the Laplace transform $E[e^{-\rho Z}]=e^{-\rho^{\alpha_{\lag+1}}}$ and set $\alpha_{L+1}=1$ then 
\begin{equation}\label{constants1}
D_{\lag}=E\big[Z^{\alpha_{\lag}}\big]\prod_{i=\lag+1}^L \Gamma\left(1-\frac{\alpha_i}{\alpha_{i+1}}\right)^{\frac{\alpha_{\lag}}{\alpha_i}}.
\end{equation}
\fi

\end{lemma}

We first finish the proof of Proposition \ref{weaklemma} using Lemma \ref{lemanara}. 

\begin{proof}[Proof of Proposition \ref{weaklemma}]
Choose the constants $D_1,\dots,D_\lag$ as in Lemma \ref{lemanara} and write the measures 
\begin{equation}
\nu_k(dx_k)=D_k\alpha_{k} x_k^{-1-\alpha_{k}},\;x_k\geq 0.
\end{equation}
We calculate the Laplace functional of $\xil$. Let $\phi_k:C_K^+(H^k)\to C_K^+(H^{k-1})$ be
\begin{equation*}
\phi_kf_k(t_1,x_1,\dots,t_{k-1},x_{k-1})=\int\int (1-e^{-f_k(t_1,x_1,\dots,t_{k-1},x_{k-1},t_k,x_k)})dt_k\nu_{k}(dx_k).
\end{equation*}
Note that $\chi_{1}^{\lag}$ is a PPP on $H$ with mean measure $dt\times D_1 \alpha_1 x_1^{-1-\alpha_1}dx_1$. Hence,
\begin{equation}\label{lapposcas1}
\Phi_{\chi_{1}^{\lag}}(f_1)=\exp\left(-\phi_1(f_1)\right).
\end{equation}
For $k=2,\dots,\lag$, using the correlation structure of $\xil$ we get
\begin{equation*}
\begin{aligned}
&E\left[\exp\bigl(-\sum_{j_1\in\N} f_1(t_{j_1},\eta_{j_1})-\cdots -\sum_{j|_k\in\N^k} f_k(t_{j_1},\eta_{j_1},\dots,t_{j|_k},\eta_{j|_k})\bigr)\Big|\chi_{1}^{\lag},\dots,\chi_{k-1}^{\lag}\right]\\&
=\exp\bigl(-\sum_{j_1\in\N} f_1(t_{j_1},\eta_{j_1})-\cdots -\sum_{\;\;j|_{k-1}\in \N^{k-1}} f_{k-1}(t_{j_1},\eta_{j_1},\dots,t_{j|_{k-1}},\eta_{j|_{k-1}})\bigr)\times
\\&\prod_{\;\;j|_{k-1}\in \N^{k-1}}E\left[\exp\bigl(-\sum_{j_k\in\N}f_k(t_{j_1},\eta_{j_1},\dots,t_{j|_{k-1}},\eta_{j|_{k-1}},t_{j|_k},\eta_{j|_k})\bigr)\Big|\chi_{k-1}^{\lag}\right]
\\&=\exp\Bigl(-\sum_{j_1\in\N} f_1(t_{j_1},\eta_{j_1})-\cdots -\sum_{\;\;j|_{k-1}\in \N^{k-1}} (f_{k-1}+\phi_kf_k)(t_{j_1},\eta_{j_1},\dots,t_{j|_{k-1}},\eta_{j|_{k-1}})\Bigr).
\end{aligned}
\end{equation*}
In the last step above we used part (i) of the description of RPC and the Laplace transform of a PPP. Taking the expectations of first and last terms in the above display gives
\begin{equation}\label{lapposcas2}
\Phi_{(\chi_{1}^{\lag},\dots,\chi_{k}^{\lag})}(f_1,\dots,f_k)=\Phi_{(\chi_{1}^{\lag},\dots,\chi_{k-1}^{\lag})}(f_1,\dots,f_{k-2},f_{k-1}+\phi_kf_k).
\end{equation}
Thus, the Laplace functional $\Phi_{\xil}$ is given by (\ref{lapposcas2}) recursively from $k=\lag$ to $k=2$ and by (\ref{lapposcas1}).

We proceed the proof by induction. Recall that $\xi_{j_1}^\nn\overset{d}=G(\lambda(J(j_1)))/c_n(1)$. Hence, since $\{J(j_1):j_1\in\N\}$ is i.i.d. so is  $\{\xi_{j_1}^\nn:j_1\in\N\}$, and
\begin{equation}
\aan{1}\Pp\bigl(\xi_{j_1}^\nn\geq u \bigr)=\frac{\aan{1}}{n}\sum_{\mu_1=1}^n \Pp_{\mu_1}^\nn\bigl((u,\infty)\bigr)=:\bar{\nu}_1^\nn(u),\; u>0.
\end{equation}
Therefore, by Lemma \ref{lemanara} and Proposition 3.21 in \cite{Re87}, for any environment in $\Omega_{1,L}$,
\begin{equation}\label{lapdyn1}
\zeta_{1}^{\nn,\lag}\Rightarrow \chi_{1}^{\lag}.
\end{equation}
Let $k\in\{2,\dots,\lag\}$ and assume that $\exists \Omega'_{k-1,L}\subseteq \Omega$ with $\mathbb{P}(\Omega'_{k-1,L})=1$ such that for any environment in $\Omega'_{k-1,L}$,
\begin{equation}\label{indcstep}
\big(\zeta_{1}^{\nn,\lag},\dots,\zeta_{k-1}^{\nn,\lag}\big)\Rightarrow \big(\chi_{1}^{\lag},\dots,\chi_{k-1}^{\lag}\big).
\end{equation}
\iffalse
for any $f_1\in C_K^+(H),\dots,f_k\in C_K^+(H^k)$
\begin{equation}
\Ee\left[\exp\Big(-\zeta_{1,\lag}^{\nn,L}(f_1)-\cdots - \zeta_{k,\lag}^{\nn,L}(f_{k})\Big)\right]\rightarrow \Phi_{(\zeta_{1,\lag},\dots,\zeta_{k,\lag})}(f_1,\dots,f_k).
\end{equation}
\fi
Let $\mathcal{G}_{k-1}^\nn$ and $\mathcal{F}^\nn_{k-1}$ be the $\sigma$-algebras
\begin{equation}
 \mathcal{G}_{k-1}^\nn=\sigma\left(\zeta_{1}^{\nn,\lag},\dots,\zeta_{k-1}^{\nn,\lag}\right),\quad \mathcal{F}^\nn_{k-1}=\sigma\left(J(j|_i):\;j|_{k-1}\in \N^{k-1},i=1,\dots,k-1\right).
\end{equation}
Observe that then,
\begin{align}
&\;\;\mathcal{E}\left[\exp\Big(-\zeta_{1}^{\nn,\lag}(f_1)-\cdots - \zeta_{k}^{\nn,\lag}(f_{k})\Big)\Big| \mathcal{G}_{k-1}^\nn\right]\\
&=\mathcal{E}\left[\mathcal{E}\Big[\exp\Big(-\zeta_{1}^{\nn,\lag}(f_1)-\cdots - \zeta_{k}^{\nn,\lag}(f_{k})\Big)\Big|\mathcal{G}_{k-1}^\nn, \F_{k-1}^\nn\Big]\Big|\mathcal{G}_{k-1}^\nn\right]\\
&\label{kocson}=\exp\left(-\zeta_{1}^{\nn,\lag}(f_1)-\cdots - \zeta_{k-1}^{\nn,\lag}(f_{k-1})-g^\nn(\zeta_{1}^{\nn,\lag},\dots,\zeta_{k-1}^{\nn,\lag}) \right),
\end{align}
where for $m=(m_1,\dots,m_{k-1})\in \M_{1,k-1}$ with marks $\{(t_{j|_i},x_{j|_i}):j|_{k-1}\in \N^{k-1},i=1,\dots,k-1\}$, $g^\nn(m)$ is given by
\begin{equation}
-\sum_{j|_k\in\N^k}\log \mathcal{E}\left[1-\int\bigl(1-e^{-f_{k}(t_{j_1},x_{j_1},\dots,t_{j|_{k-1}},x_{j|_{k-1}},j_{k}/\aan{k},x_k)}\bigr)\frac{\nu_k^\nn(dx_k;J(j|_{k-1}))}{\aan{k}}\right].
\end{equation}
By Lemma \ref{lemanara}, for any environment in $\Omega_{k,L}$ where $\Omega_{k,L}$ is as given in Lemma \ref{lemanara}, for any $m\in\M_{1,k-1}$, 
\begin{equation}
g^\nn(m)\to \sum_{\;\;j|_{k-1}\in \N^{k-1}}\phi_k f_k(t_{j_1},x_{j_1},\dots,t_{j|_{k-1}},x_{j|_{k-1}})=m_{k-1}(\phi_kf_k).
\end{equation}
By Proposition 3.16 in \cite{Re87}, we have for any $A\subseteq \M_{1,k-1}$ relatively compact and $F\subseteq H^{k-1}$ also relatively compact,
\begin{equation}
\sup_{m\in A}\sum_{\;\;j|_{k-1}\in \N^{k-1}}\mathds{1}\left\{(t_{j_1},x_{j_1},\dots,t_{j|_{k-1}},x_{j|_{k-1}})\in F\right\}<\infty.
\end{equation}
This, together with Lemma \ref{lemanara}, yield 
\begin{equation}
\sup_{m\in A}\parallel g^\nn(m)-m_{k-1}(\phi_kf_k) \parallel \to 0.
\end{equation}
Hence, $g^\nn(m)\to m_{k-1}(\phi_kf_k)$ uniformly on compact sets. Also, it is obvious that $\parallel \exp(-g^\nn(m))\parallel\leq 1$ for all $m\in \M_{1,k-1}$. Hence, by taking the expectation of (\ref{kocson}) and using the induction step (\ref{indcstep}), we conclude that for any environment in $\Omega'_{k,L}=\Omega_{k-1,L}'\cap \Omega_{k,L}$, as $n\to\infty$ 
\begin{equation}
\begin{aligned}
\mathcal{E}&\left[\exp\Big(-\zeta_{1}^{\nn,\lag}(f_1)-\cdots - \zeta_{k}^{\nn,\lag}(f_{k})\Big)\right]\\&\rightarrow \Phi_{(\chi_{1}^{\lag},\dots,\chi_{k-1}^{\lag})}(f_1,\dots,f_{k-2},f_{k-1}+\phi_{k}f_{k}).
\end{aligned}
\end{equation}
This and (\ref{lapdyn1}) finish the proof of Proposition \ref{weaklemma} with $\Omega'_{L}=\Omega'_{\lag,L}$. 

\end{proof}

\begin{proof}[Proof of Lemma \ref{lemanara} part (i)]
We first prove that for any $u>0$, $\exists \Omega_{k,L}^u\subset \Omega$ with $\mathbb{P}(\Omega_{k,L}^u)=1$ such that for every environment in $\Omega_{k,L}^u$, uniformly in $\mu|_{k-1}$, as $n\to\infty$ $
\bv{k}\rightarrow \Gamma(1+\alpha_k) u^{-\alpha_k}.
$
We first calculate the expectation of $\bv{k}$ over the environment. Note that under $\mathbb{P}$, $\Big\{\Pp^\nn_{\mu|_k}(u,\infty):\mu|_{k}\in V|_k\Big\}$ is an i.i.d. collection. Hence, using (\ref{distlam}) and (\ref{gdist}) we have
\begin{align*}
\mathbb{E}\big[\bv{k}\big]&=\aan{k}\mathbb{E}\left[\big(1-\lambda(\mu|_k)\big)^{\luc}\right]\\&=\aan{k}\int_0^1(1-x)^{\lfloor u\ccn{k}\rfloor}x^{\alpha_k-1}dx\\&=
\aan{k}\int_0^{\infty} C_k^\nn(z) \alpha_k z^{-1+\alpha_k} u^{-\alpha_k} (\ccn{k})^{-\alpha_k} dz
\end{align*}
where we performed the change of variables $x\to u\ccn{k}/z$ and set 
\begin{equation}
 C_k^\nn(z)=\left(1-\frac{z}{u\ccn{k}}\right)^\luc\inc{z\leq u\ccn{k}}. 
\end{equation}
Since $k<\lag$, by the scaling relations in (\ref{scalrel}), we have $\aan{k}=(\ccn{k})^{\alpha_k}$. Using the bounds $u\ccn{k}-1<\luc\leq u\ccn{k}$ and the fact that $\ccn{k}$ is diverging, we have for all $z\geq 0$, $C_k^\nn(z)\to e^{-z}$ and $C_k^\nn(z)\leq e^{-z}$. Hence,
 by the dominated convergence theorem, we can conclude that
\begin{equation}\label{cisiki}
\mathbb{E}[\bv{k}]\longrightarrow u^{-\alpha_k}\alpha_k \Gamma(\alpha_k)=u^{-\alpha_k}\Gamma(1+\alpha_k).
\end{equation}
To control the fluctuations of $\bv{k}$ we write, for $\theta_n\geq 0$,
\begin{align}
\nonumber\mathbb{P}&\Big(\exists \mu|_{k-1}\in V|_{k-1}:\; \big|\bv{k}-\mathbb{E}[\bv{k}]\big|\geq \theta_n\Big)\\&\label{piyano}\leq n^{k-1} \mathbb{P}\Big( \big|\bv{k}-\mathbb{E}[\bv{k}]\big|\geq \theta_n\Big),
\end{align}
and bound the above probabilities using Bennett's bound (see \cite{Ben62}), which states that if $(X(\mu):\mu=1,\dots,n)$ is a family of centered i.i.d. random variables that satisfies $\max_{\mu=1,\dots,n}|X(\mu)|\leq a$, then for any $b^2\geq \sum_{\mu=1}^n \mathbb{E} X^2(\mu)$ and $t<b^2/(2a)$,
\qw\label{bune}
\pp\left(\Big|\sum_{\mu=1}^n X(\mu)\Big|>t\right)\leq \exp\Big(-\frac{t^2}{4b^2}\Big).
\qwe
For fixed $\mu|_{k-1}$, let $\{X(\mu_k):\mu_k=1,\dots,\mmn{k}\}$ be the collection of i.i.d. random variables given by
\begin{equation}
X(\mu_k)=\Pp\Big(G(\lambda(\mu|_k))\geq u\ccn{k}\Big)-\mathbb{E}\left[\Pp\Big(G(\lambda(\mu|_k))\geq u\ccn{k}\Big)\right].
\end{equation}
Then, $\max |X(\mu_k)|\leq 2$ and $\sum_{\mu_k}\mathbb{E}[X^2(\mu_k)]\leq C \mmn{k}/\aan{k}$ for $n$ large enough, for some positive constant $C$. Hence, we can choose $b^2=C\mmn{k}/\aan{k}$ and get that (\ref{piyano}) is bounded above by
\begin{equation}
n^{k-1} \exp\left(-C'\frac{\mmn{k}\theta_n^2}{\aan{k}}\right),
\end{equation}
provided that $\theta_n<C/2$. Since $k<\lag$, the ratio $\mmn{k}/\aan{k}$ diverges at least polynomially fast and thus, we can choose $\theta_n$ such that $\lim_{n\to\infty}\theta_n=0$ and $\mmn{k}\theta_n^2/\aan{k}$ diverges at least polynomially fast. Hence,
\qw\label{sune}
\mathbb{P}\Big(\exists \mu|_{k-1}\in V|_{k-1}^\nn:\; \big|\bv{k}-\mathbb{E}[\bv{k}]\big|\geq \theta_n\Big)\leq \exp(-n^{c})
\qwe
for some $c>0$. Borel-Cantelli Lemma and (\ref{cisiki}) prove that $\mathbb{P}(\Omega_{k,L}^u)=1$, where for any $u>0$, $\Omega_{k,L}^u$ is the set of environments where $\bar{\nu}_k^\nn(u;\mu|_{k-1})\to \bar{\nu}_k(u)$, uniformly in $\mu|_{k-1}\in V|_{k-1}$. Let $\mathbb{Q}^+$ be the set of positive rational numbers. The fact that $\bar{\nu}^\nn(u,\mu|_{k-1})$ is monotone, $\bar{\nu}_k^\nn(u)$ is continuous and the countable intersection of probability one events is also a probability one event finish the proof with 
$\Omega_{k,L}=\cap_{u\in \mathbb{Q}^+}\Omega_{k,L}^u.$
\end{proof}
\begin{proof}[Proof of Lemma \ref{lemanara} part (ii)]

For $m=1,\dots,L$ and $\mu|_{m-1}\in V|_{m-1}$, we define random variables $\Theta_m(\cdot;\mu|_{m-1})$ recursively, from $m=L$ to $m=1$, as follows. Let 
\begin{equation}\label{thetal1}
\Theta_{L}(r;\mu|_{L-1})=\sum_{j=1}^r\lambda^{-1}(\mu|_{L-1}J(j))e_j,\quad r\in\N,
\end{equation}
where the independent collections $\{J(j):j\in\N\}$ and $\{e_j:j\in\N\}$ are i.i.d. with uniform distribution on $[n]$ and the distribution of a mean one exponential random variable, respectively. Also, let
\begin{equation}\label{thetam}
\Theta_m(r;\mu|_{m-1})=\sum_{j=1}^r \Theta^{(j)}_{m+1}\bigl(G\bigl(\lambda\bigl(\mu|_{m-1}J(j)\bigr)\bigr);\mu|_{m-1}J(j)\bigr)
\end{equation}
where $\{J(j):j\in\N\}$ is i.i.d. with uniform distribution on $[n]$, and given this collection, $\{\Theta_{m+1}^{(j)}(\cdot;\mu|_{m-1}J(j)):j\in\N\}$ and $\{G(\lambda(\mu|_{m-1}J(j))):j\in\N\}$ are independent  collections of independent random variables with each $\Theta^{(j)}_{m+1}(\cdot;\mu|_{m-1}J(j))$ having the distribution of $\Theta_{m+1}(\cdot;\mu|_{m-1}J(j))$, and $G(\lambda(\mu|_{m-1}J(j)))$ having the distribution of a geometric random variable with probability of success $\lambda(\mu|_{m-1}J(j))$.  In words, $\Theta_m(\cdot;\mu|_{m-1})$ has the distribution of the first level clock process of the GREM-like trap model, reduced to the subtree attached to $\mu|_{m-1}$. Note that the distribution of $\Theta_{m}(\cdot;\mu|_{m-1})$ is i.i.d. in $\mu|_{m-1}$. Finally, for $m=\lag+1,\dots,L$ and $\mu|_{m-1}\in V|_{m-1}$, we define
\begin{equation}\label{scaltheta}
\Theta_{m}^\nn(r;\mu|_{m-1})=\frac{\Theta_m(\lfloor r \aan{m} \rfloor;\;\mu|_{m-1})}{\tn},\quad r>0.
\end{equation}

\begin{lemma}\label{propcasc} For any $m=\lag+1,\dots,L$, $r>0$ and $u>0$, 
\begin{equation}
\mathbb{E}\Big[\Pp\big(\Theta_m^\nn(r;\mu|_{m-1})\geq u\big)\Big]\longrightarrow P(r Z_m\geq u), 
\end{equation}
where $Z_m$ is a positive random variable whose Laplace transform is given by
\begin{equation}\label{lapz}
E[e^{-\kappa Z_m}]=e^{-d_m \kappa^{\alpha_{m}}},\;\;\;\kappa>0
\end{equation}
for some $d_m>0$. 
\end{lemma}
\begin{proof}[Proof of Lemma \ref{propcasc} ]
Note that it is enough to show that for any $r>0$ and $\kappa>0$
\begin{equation}\label{bsg}
\mathbb{E}\left[\Ee\Big[\exp\big({-\kappa \Theta_{m}^\nn(r;\mu|_{m-1})}\big)\Big]\right]\longrightarrow e^{-d_m(\kappa r)^{\alpha_{m}}}.
\end{equation}
We use the following fact several times in the proof: Let $X_n,X$ be random variables on $[0,1]$ and $a_n$ be a diverging scale, then
\begin{equation}\label{edi}
E[X_n^{r a_n}]\underset{n\to\infty}\longrightarrow E[e^{-rX}] \;\;\forall r>0 \Leftrightarrow E[e^{-r a_n (1-X_n)}]\underset{n\to\infty}\longrightarrow E[e^{-rX}]\;\;\forall r>0.
\end{equation}
We proceed the proof by induction. We first prove the case where $m=L$. By (\ref{thetal1}) and (\ref{scaltheta}),
\begin{equation}
\Ee\Big[\exp\big({-\kappa {\Theta_{L}^\nn(r;\;\mu|_{L-1})}}\big)\Big]=\left\{\sum_{\mu_L=1}^{n}\frac{1}{\mmn{L}}\frac{1}{1+\kappa \lambda^{-1}(\mu|_L)/\tn}\right\}^{\lfloor r\aan{L}\rfloor}.
\end{equation}
By (\ref{scalrel}), we have $\aan{L}\gg \mmn{L}$ and $c_n=\ccn{L}=\aan{L}n^{1/\alpha_L-1}$. Thus,
\begin{equation*}
\begin{aligned}
&\mmn{L}\E\left[1-\exp\left(-\frac{\aan{L}}{\mmn{L}} \frac{r \kappa \lambda^{-1}(\mu|_L)/\ccn{L}}{1+\kappa\lambda^{-1}(\mu|_L)/\ccn{L}}\right)\right]\\&
%=\mmn{L}\int_1^\infty \left\{1-\exp\left(-\frac{\aan{L}}{\mmn{L}} \frac{r \kappa x/\ccn{L}}{1+\kappa x/\ccn{L}}\right)\right\}\alpha_{L}x^{-1-\alpha_L}dx\\&
=\int_{n^{-1/{\alpha_L}}}^\infty \left\{1-\exp\left(- \frac{r \kappa y}{1+\kappa y \mmn{L}/\aan{L}}\right)\right\}\alpha_{L}y^{-1-\alpha_L}dy\\&
\underset{n\to\infty}\longrightarrow \int_0^{\infty}\Big(1-\exp\left(- {r \kappa y}\right)\Big)\alpha_{L}y^{-1-\alpha_L}dy=\Gamma(1-\alpha_L)(r\kappa)^{\alpha_L}.
\end{aligned}
\end{equation*}
In taking the above limit we used $\mathbb{P}(\lambda^{-1}(\mu|_L)\geq u)=u^{-\alpha_L}$, the bound $1-e^{-cy}\leq 1\wedge cy$ and the dominated convergence theorem. By the independence structure of the landscape, 
\begin{align}
\nonumber &\E\left[\exp \left(-\sum_{\mu_L=1}^n\frac{\aan{L}}{\mmn{L}}\frac{r\kappa\lambda^{-1}(\mu|_L)/\ccn{L}}{1+\kappa\lambda^{-1}(\mu|_L)/\ccn{L}}\right)\right]\\\label{eseginski}&=
\left(1-\E\left[1-\exp \left(-\frac{\aan{L}}{\mmn{L}} \frac{r\kappa\lambda^{-1}(\mu|_L)/\ccn{L}}{1+\kappa\lambda^{-1}(\mu)/\ccn{L}}\right)\right]\right)^\mmn{L}\\ \nonumber &\underset{n\to\infty}\longrightarrow\exp\Big(-\Gamma(1-\alpha_L) (r\kappa)^{\alpha_L}\Big).
\end{align}
Finally, using (\ref{edi}) with $X_n=\sum_{\mu_L}\frac{1}{\mmn{L}}\frac{1}{1+\kappa \lambda^{-1}(\mu|_L)/\ccn{L}}$ and $a_n=\aan{L}$ finish the proof of (\ref{bsg}) for $m=L$, where $d_{L}=\Gamma(1-\alpha_L)$. 

Now assume that (\ref{bsg}) is true for $m+1$. By (\ref{thetam}), we have
\begin{equation}
\Ee\left[\exp\big(-\kappa \Theta_m(1;\mu|_{m-1})\big)\right]=\sum_{\mu_{m}=1}^n\frac{1}{\mmn{m+1}}\Ee\left[\exp\big(-\kappa \Theta_{m+1}(G(\lambda(\mu|_{m}));\mu|_{m})\big)\right].
\end{equation}
Using the geometric distribution and (\ref{thetam}) we get
\begin{equation*}
\Ee\left[\exp\big(-\kappa \Theta_{m+1}(G(\lambda(\mu|_{m}));\mu|_{m})\big)\right]=
\frac{1}{1+\lambda^{-1}(\mu|_{m})\varphi_{m+1}(\kappa;\mu|_{m})},
\end{equation*}
where
\begin{equation}\label{kabiz}
\varphi_{m+1}(\kappa;\mu|_{m})=\frac{1-\Ee[e^{-\kappa \Theta_{m+1}(1;\mu|_{m})} ]}  {\Ee[e^{-\kappa \Theta_{m+1}(1;\mu|_{m})} ]}.
\end{equation}
Hence,
\begin{equation*}
\Ee\Big[\exp\big({-\kappa {\Theta_m^\nn( r;\;\mu|_{m-1})}}\big)\Big]=\left\{\sum_{\mu_{m}}\frac{1}{\mmn{m+1}}\frac{1}{1+\lambda^{-1}(\mu|_{m})\varphi_{m+1}(\kappa/\tn;\mu|_{m})}\right\}^{\lfloor r\aan{m}\rfloor}.
\end{equation*}
By the induction step, for any $\kappa,r>0$,
\begin{equation}
\mathbb{E}\Big[\Ee[e^{-\frac{\kappa}{\tn}\Theta_{m+1}(1;\mu|_{m})}]^{\lfloor r\aan{m+1} \rfloor}\Big]\longrightarrow E[e^{-\kappa r Z_{m+1}}],
\end{equation}
and consequently,
\begin{equation}\label{fena}
\mathbb{E}\Big[\exp(-r\aan{m+1} \varphi_{m+1}(\kappa/\tn;\mu|_{m} )\Big]\longrightarrow E[e^{-r\kappa Z_{m+1}}]=e^{-d_{m+1} (r\kappa)^{\alpha_{m+1}}}.
\end{equation}
We calculate
\begin{align*}
&\mmn{m+1}\mathbb{E}\left[1-\exp\left(-\frac{\aan{m}}{\mmn{m+1}}\frac{r\lambda^{-1}(\mu|_{m})\varphi_{m+1}(\kappa/\tn;\mu|_{m})}{1+\lambda^{-1}(\mu|_{m})\varphi_{m+1}(\kappa/\tn;\mu|_{m})}\right)\right]
\\&=\mathbb{E}\left[\mmn{m+1}\int_1^\infty \left(1-\exp(-\frac{\aan{m}}{\mmn{m+1}}\frac{rx\varphi_{m+1}(\kappa/\tn;\mu|_{m})}{1+x\varphi_{m+1}(\kappa/\tn;\mu|_{m})})\right)\alpha_{m}x^{-1-\alpha_{m}}dx\right].
\end{align*}
Using the change of variables $x=y n^{1/\alpha_{m}}$ and $\aan{m+1}=\ccn{m}=\aan{m} n^{1/\alpha_m-1}$ we get that the above display is equal to
\begin{equation}\label{koktan}
\begin{aligned}
\int_{n^{-1/\alpha_{m}}}^\infty \mathbb{E}&\left[1-\exp\left(-\frac{ry\varphi_{m+1}(\kappa/\tn;\mu|_{m}) \aan{m+1}  }     {1+y\varphi_{m+1}(\kappa/\tn;\mu|_{m}) \aan{m+1}\mmn{m+1}/\aan{m}}     \right)\right]\alpha_{m+1}y^{-1-\alpha_{m+1}}dy.
\end{aligned}
\end{equation}
Since $\alpha_{m}<\alpha_{m+1}$ and $\mmn{m}\ll \aan{m}$, by (\ref{fena}), we can conclude that the sequence in (\ref{koktan}) converges to 
\begin{equation}
\int_{0}^\infty \Big(1-\exp\bigl(-d_{m+1}(ry\kappa)^{\alpha_{m+1}}\bigr)\Big)\alpha_m y^{-1-\alpha_m}dy=d_m (r\kappa)^{\alpha_m}
\end{equation}
where $
d_m=d_{m+1}^{\alpha_{m}/\alpha_{m+1}}\Gamma(1-\frac{\alpha_{m}}{\alpha_{m+1}})$. A calculation as in (\ref{eseginski}) and (\ref{edi}) finish the proof.

\end{proof}
As before, we have $
\mathbb{E}[\bv{\lag}]=\aan{\lag}\mathbb{E}[\Pp_{\mu|_{\lag}}^\nn((u,\infty))]
$. Note that the distribution of $\Lambda(j|_{\lag})$, when $J(j|_{\lag})=\mu|_{\lag}$, is that of $\Theta_{\lag+1}\bigl(G(\lambda(\mu|_{\lag}));\mu|_\lag\bigr)$ where $\Theta_{\lag+1}(\cdot;\mu|_{\lag})$ and $G(\lambda(\mu|_{\lag}))$ are independent. Hence, by (\ref{pmuldef}),
\begin{equation}
\Pp_{\mu|_{\lag}}^\nn\bigl((u,\infty)\bigr)=\sum_{i=1}^\infty \lambda(\mu|_{\lag})\big(1-\lambda(\mu|_{\lag})\big)^{i-1}\Pp(\Theta_{\lag+1}(i;\mu|_{\lag})\geq u \tn).
\end{equation}
Thus,
\begin{equation}\label{araz1}
\mathbb{E}[\bv{\lag}]=\sum_{i=1}^\infty g_{n,u}\left(\frac{i}{\ccn{\lag}}\right)h_n\left(\frac{i}{\ccn{\lag}}\right)\frac{1}{\ccn{\lag}}
\end{equation}
where
\begin{equation}
g_{n,u}(r)=\mathbb{E}\Big[\Pp\big(\Theta_{\lag+1}^\nn(r;\mu|_{\lag})\geq u\big)\Big],
\end{equation}
and
\begin{equation}
h_n(r)=\aan{\lag}\ccn{\lag}\mathbb{E}\left[\lambda(\mu|_{\lag})\big(1-\lambda(\mu|_{\lag})\big)^{\lfloor r\ccn{\lag}\rfloor -1}\right].
\end{equation}
By Proposition \ref{propcasc} we get
\begin{equation}\label{araz2}
g_{n,u}(r)\longrightarrow g_u(r):=P(rZ_{\lag+1}\geq u),
\end{equation}
and a simple calculation yields
\begin{equation}\label{araz3}
h_n(r)\longrightarrow \alpha_{\lag}r^{-1-\alpha_{\lag}} \Gamma(1+\alpha_{\lag}).
\end{equation}
Using (\ref{lapz}) and a Tauberian theorem (see e.g. Corollary 8.1.7 in \cite{Bingham}) we can conclude that there exists a $C>0$ such that for all $r$ small enough, 
\begin{equation}\label{araz4}
g_u(r)\leq C r^{\alpha_{\lag+1}} u^{-\alpha_{\lag+1}}.
\end{equation}
Via (\ref{araz1}) and (\ref{araz2})-(\ref{araz4}),
\begin{equation}
\mathbb{E}[\bv{\lag}]\longrightarrow \Gamma(1+\alpha_{\lag})\int_0^\infty g_u(r)\alpha_{\lag}r^{-1-\alpha_{\lag}}dr.
\end{equation}
We use Bennett's bound once again to finish the proof. For fixed $\mu|_{\lag-1}$, consider the collection $\{X_{\mu_{\lag}}:\mu_{\lag}\in \mathcal{M}_{\lag}\}$ where
\begin{equation}
X_{\mu_{\lag}}=\Pp_{\mu|_{\lag}}^\nn\bigl((u,\infty)\bigr)-\mathbb{E}\Big[\Pp_{\mu|_{\lag}}^\nn\bigl((u,\infty)\bigr)\Big].
\end{equation}
Clearly, $\max_{\mu_{\lag}}|X_{\mu_{\lag}}|\leq 2$, and using Jensen's inequality we get
\begin{equation}
\sum_{\mu_{\lag}}\E[X_{\mu_{\lag}}^2]\leq C \frac{\mmn{\lag}}{\aan{\lag}}.
\end{equation}
Hence, since $\aan{\lag}\ll\mmn{\lag}$, we can proceed exactly as in the proof of Lemma \ref{lemanara} part (i) to conclude that there exists $\Omega_{\lag,L}\subset \Omega$ with $\mathbb{P}(\Omega_{\lag,L})=1$ such that for any environment in $\Omega_{\lag,L}$, for any $u>0$, uniformly in $\mu|_{\lag-1}$
\begin{equation}
\bv{\lag}\longrightarrow \Gamma(1+\alpha_{\lag})\int_0^\infty g_u(r)\alpha_{\lag}r^{-1-\alpha_{\lag}}dr.
\end{equation}
We have 
\begin{align}
\int_0^\infty g_u(r)\alpha_{\lag}r^{-1-\alpha_{\lag}}dr&=\int_0^\infty P(Z_{\lag+1}\geq u/r)\alpha_{\lag}r^{-1-\alpha_{\lag}}dr\\&=u^{-\alpha_{\lag}} \int_0^\infty P(Z_{\lag+1}\geq s)\alpha_{\lag}s^{-1+\alpha_{\lag}}ds.
\end{align}
Since $\alpha_{\lag}<\alpha_{\lag+1}$, by the Laplace transform, (\ref{lapz}), of $Z_{\lag+1}$,
\begin{equation}
D_\lag=\Gamma(1+\alpha_\lag)\int_0^\infty P(Z_{\lag+1}\geq s)\alpha_{\lag}s^{\alpha_{\lag}-1}ds<\infty.
\end{equation}
Thus, we are finished with the proof of Lemma \ref{lemanara} part (ii). 
\end{proof}

\subsection{Continuity of functionals on the space cascade of point measures} In this subsection we prove that functionals used to define the clock processes are, after certain truncations, continuous.
\newcommand{\wcm}{\widetilde{\mathcal{M}}}

\newcommand{\g}{{(\gamma)}}
\newcommand{\gam}{{(\gamma,\gamma^{-1})}}
\newcommand{\twm}{{\widetilde{\mathcal{M}}}}

For $\gamma>0$, let $\M^\g_{1,l}$ be the subset of $\M_{1,l}$ such that for $m\in\M^\g_{1,l}$, whose marks are given by
\begin{equation}\label{collpoint}
\left\{(t_{j|_k},x_{j|_k}):j|_l\in \N^l,k=1,\dots,l\right\},
\end{equation}
it holds true that
\begin{equation}\label{finite1}
m_{1}\big((0,\infty)\times \{\gamma,\gamma^{-1}\}\big)=0,
\end{equation}
and for $k=1,\dots,l-1$ and $j|_k\in\N^k$,
\begin{equation}\label{finitek}
m_{k+1}\Big(\{(t_{j_1},x_{j_1},\dots,t_{j|_k},x_{j|_k})\}\times \partial \big( (0,x_{j|_k}]\times (\gamma,\gamma^{-1})\big)\Big)=0.
\end{equation}
Here, $\partial$ denotes the boundary of a set. For $m\in\mathcal{M}_{1,l}$, let  $m^\g=(m_1^\g,\dots,m_l^\g)$ be
\begin{equation}
m_k^\g=\sum_{j|_k\in\N^k} \vep_{(t_{j_1},x_{j_1},\dots,t_{j|_k},x_{j|_k})}\prod_{i=1}^k\mathds{1}\{x_{j|_i}\in\gam\}.
\end{equation}
Note that then, $m^\g\in\wt{\M}_{1,l}$. We introduce the maps $T^\g_l: \M_{1,l}\to D([0,\infty))$ and $\overline{T}_l^{\g}:\M_{1,l}\to \wt{\M}_{1,l-1}$ by
\begin{equation}
T^\g_{l}(m)=T_{l}(m^\g),\quad \text{   and   }\quad\overline{T}_l^\g(m)=\overline{T}_l(m^\g). 
\end{equation}

\iffalse

%HERE THE OLD TRADITIONAL DEFINITION OF CLOCK PROCESSES

For any $j|_k$ we define a mapping $F_{j|_k}:\twm_l^\g\to (0,\infty)$ 
\begin{equation}
 F_{j|_k}(m)=\sum_{t^\g_{j|_k j_{k+1}}\leq x^\g_{j|_k}}\cdots \sum_{\;\;\;t^\g_{j|_{l-1}j_l}\leq x^\g_{j|_{l-1}}} x^\g_{j|_l}. 
\end{equation}
For $k=1,\dots,l$, $r=1,\dots,k$ and $j|_{r}$ we define a mapping $I^k_{r,j|_r}:M^\g_l\to \N$ by
\begin{equation}
 I^k_{r,j|_r}(m)=\sum_{t^\g_{j|_{r+1}}\leq x_{j|_r}}\cdots \sum_{t^\g_{j|_{k}}\leq x^\g_{j|_{k-1}}} 1,\quad r<k,
\end{equation}
and 
\begin{equation}
 I^k_{k,j|_k}(m)=1.
\end{equation}
Let $\{X_m:\;m\in\N\}$ be a sequence of positive numbers, and we set $X_0= 0$. Let $z\geq 0$. For $t\geq z$ we define
\begin{equation}
 g(t)=\max\{m: z+X_0+X_1+\cdots+X_m \leq t \}
\end{equation}
and 
\begin{equation}
 h(t)=z+X_0+X_1+\cdots X_{g(t)}.
\end{equation}
We call $g$ the counting function with inter arrival times (c.f.w.i.a.t.) $\{X_m:m\in\N\}$ on $[z,\infty)$, and $h$ its dual.

Let $g_1^k$ be the c.f.w.i.a.t $\{I_{1,j_1}^k(m^\g):j_1\in\N\}$ on $[0,\infty)$ and $h_1^k$ be its dual. From $r=2$ to $r=k$, we recursively define $g_r^k$ as the c.f.w.i.a.t. $\{I_{r,j|_{r-1}j_r}^k(m^\g):j_r\in\N\}$ on $[h_1^k(t)+\cdots+ h_{r-1}^k(t),\infty)$, where
\begin{equation}
 j|_{r-1}=(g_1^k(t)+1)\cdots (g_{r-1}^k(t)+1),
\end{equation}
and we denote its dual by $h_r^k$. 

\fi

%HERE THE OLD TRADITIONAL DEFINITION OF CLOCK PROCESSES

\begin{lemma}\label{contlem}
For any $\gamma>0$, the maps $T_{l}^\g$ and $\overline{T}_l^\g$ are continuous on $ \mathcal{M}^\g_{1,l}$. 
\end{lemma}

We use the following proposition which is also a generalization from the usual point processes. 
\begin{proposition}\label{gozde}
Let $A_1,\dots,A_l\subset H$ be compact sets with $$m_1(\partial A_1)=\cdots=m_l(\partial (A_1\times\cdots\times A_l))=0.$$ Let $m^\nn,m\in \M_{1,l}$ be such that $m^{\nn}\underset{n\to\infty}\longrightarrow m$. Then, for any $k=1,\dots,l$, after relabeling,
\begin{equation*}
m_k\Big(\cdot \cap (A_1\times \cdots\times A_k)\Big)=\sum_{j_1=1}^{q_1}\;\sum_{j_2=1}^{q_2(j_1)}\;\;\cdots \;\sum_{j_k=1}^{q_k(j|_{k-1})}\vep_{(t_{j_1},x_{j_1},\dots,t_{j|_{k}},x_{j|_{k}})}
\end{equation*}
and
\begin{equation*}
m_k^\nn\Big(\cdot \cap (A_1\times \cdots\times A_k)\Big)=\sum_{j_1=1}^{q_1}\;\sum_{j_2=1}^{q_2(j_1)}\;\;\cdots \;\sum_{j_k=1}^{q_k(j|_{k-1})}\vep_{(t^\nn_{j_1},x^\nn_{j_1},\dots,t^\nn_{j|_{k}},x^\nn_{j|_{k}})}
\end{equation*}
for all $n\geq n(A_1,\dots,A_l)$. Moreover, for any $r=1,\dots,k$; $j_1=1,\dots,q_1$; $j_2=1,\dots,q_2(j_1);\dots;$ $j_r=1,\dots,q_r(j|_{r-1})$, as $n\to\infty$
\begin{equation*}
\parallel (t_{j|_r}^\nn,x_{j|_r}^\nn)-(t_{j|_r},x_{j|_r})\parallel \longrightarrow 0.
\end{equation*}
\end{proposition}
\begin{proof}[Proof of Proposition \ref{gozde}]

When $k=1$ we have the usual point processes case and it is covered by Proposition 3.13 in \cite{Re87}. We now consider $k=2$. Via, once again, Proposition 3.13 in \cite{Re87}, there exists a $q_1\in\N$ such that, after relabeling, for any $n\geq n(A_1)$ 
\begin{equation}\label{starr}
m_1(\cdot\cap A_1)=\sum_{j_1=1}^{q_1} \vep_{(t_{j_1},x_{j_1})},\quad m_1^\nn(\cdot\cap A_1)=\sum_{j_1=1}^{q_1} \vep_{(t_{j_1}^\nn,x_{j_1}^\nn)},
\end{equation}
and
\begin{equation}
\parallel (t_{j_1}^\nn,x_{j_1}^\nn)- (t_{j_1},x_{j_1})\parallel \to 0,\quad \forall j_1=1,\dots,q_1.
\end{equation}
Let $B_{r}(z)$ and $\bar{B}_r(z)$ denote the open and closed ball, respectively, around $z\in H^k$ with radius $r$. We choose $\epsilon>0$ small enough so that $\bar{B}_{\epsilon}(t_1,x_1),\dots,\bar{B}_{\epsilon}(t_{q_1},x_{q_1})$ are disjoint and for any $j_1=1,\dots,q_1$, we have  $m_1(\bar{B}_\epsilon((t_{j_1},x_{j_1})))=1$ and $\bar{B}_\epsilon(t_{j_1},x_{j})\subseteq A_1^o$, where $A_1^o$ denotes the interior of $A_1$ (recall that we assume that all the point measures are simple). Hence, for any $j_1=1,\dots,q_1$ we can find a $q_2(j_1)$ such that, after relabeling,
\begin{equation}\label{starr2}
m_2(\cdot\cap \bar{B}_{\epsilon}(t_{j_1},x_{j_1})\times A_2)=\sum_{j_2=1}^{q_2(j_1)}\vep_{(t_{j_1},x_{j_1},t_{j_1j_2},x_{j_1j_2})}. 
\end{equation}
We choose $n$ large enough so that for any $j_1=1,\dots,q_1$, we have $\parallel (t_{j_1}^\nn,x_{j_1}^\nn)-(t_{j_1},x_{j_1})\parallel<\epsilon/2$. Thus, by (\ref{starr}), (\ref{starr2}) and Proposition 3.13 in \cite{Re87}, after relabeling, for $n\geq n(\bar{B}_{\epsilon}(t_{j_1},x_{j_1})\times A_2)$,
\begin{equation}
m_2^\nn(\cdot\cap \bar{B}_{\epsilon}(t_{j_1},x_{j_1})\times A_2)=\sum_{j_2=1}^{q_2(j_1)}\vep_{(t_{j_1}^\nn,x_{j_1}^\nn,t_{j_1j_2}^\nn,x_{j_1j_2}^\nn)}
\end{equation}
and
\begin{equation}\label{starr3}
\parallel (t_{j_1j_2}^\nn,x_{j_1j_2}^\nn)-(t_{j_1j_2},x_{j_1j_2})\parallel \to 0,\quad \forall j_2=1,\dots,q_2(j_1).
\end{equation}
Also, by (\ref{starr})-(\ref{starr3}), for $n\geq \max_{j_1=1,\dots,q_1}n(\bar{B}_{\epsilon}(t_{j_1},x_{j_1})\times A_2)=n(A_1,A_2)$,
\begin{equation}
m_2(\cdot\cap A_1\times A_2)=\sum_{j_1=1}^{q_1}\sum_{j_2=1}^{q_2(j_1)}\vep_{(t_{j_1},x_{j_1},t_{j_1j_2},x_{j_1j_2})}
\end{equation}
and
\begin{equation}
m_2^\nn(\cdot\cap A_1\times A_2)=\sum_{j_1=1}^{q_1}\sum_{j_2=1}^{q_2(j_1)}\vep_{(t_{j_1}^\nn,x_{j_1}^\nn,t_{j_1j_2}^\nn,x_{j_1j_2}^\nn)}.
\end{equation}
This proves the case $k=2$. Iterating the exact same procedure finishes the proof. 
\end{proof}

\begin{proof}[Proof of Lemma \ref{contlem}]

Let $m\in\mathcal{M}^\g_{1,l}$ and $m^\nn\in \M_{1,l}$ with $m^\nn\to m$ as $n\to\infty$.  We first prove that $T_{l}^\g(m^\nn)\to T_{l}^\g(m)$ on $D([0,\infty))$ equipped with the Skorohod $J_1$ topology. It is enough to show that $T_{l}^\g(m^\nn)\to T_{l}^\g(m)$ on $D([0,t'])$ for any continuity point $t'$ of $T_{l}^\g(m)$, which is equivalent to $m_{1}(\{t'\}\times (\gamma,\gamma^{-1}))=0$. Since $m\in \mathcal{M}_{1,l}^\g$, for any such $t'$ there exists a $q_1$ such that, after relabeling,
\begin{equation}
m_1\bigl(\cdot \cap (0,t']\times \gam \bigr)=\sum_{j_1=1}^{q_1}\vep_{(t_{j_1},x_{j_1})}
\end{equation}
Recursively, from $k=2$ to $k=l$, for any $j_1=1,\dots,q_1;$ $j_2=1,\dots,q_2(j_1);\dots;$ $j_{k-1}=1,\dots,q_{k-1}(j|_{k-2})$, there exists a $q_{k}(j|_{k-1})$ such that
\begin{equation}
\begin{aligned}
m_k&\left(\cdot \;\cap\;\bigl(t_{j_1},x_{j_1},\dots,t_{j|_{k-1}},x_{j|_{k-1}}\bigr)\times \bigl((0,x_{j|_{k-1}}]\times (\gamma,\gamma^{-1})\bigr)\right)\\&=\sum_{j_k=1}^{q_k(j|_{k-1})}\mathds{1}\left\{(t_{j_1},x_{j_1},\dots,t_{j|_k},x_{j|_k})\in \cdot\right\}.
\end{aligned}
\end{equation}
By (\ref{finitek}), we can find a $\delta>0$ small enough so that, for any $k=1,\dots,l-1$ and $j_1=1,\dots,q_1$; $j_2=1,\dots,q_2(j_1);\dots;$ $j_{k}=1,\dots,q_{k}(j|_{k-1})$,
\begin{equation}\label{noboundary}
m_{k+1}\Big(\{(t_{j_1},x_{j_1},\dots,t_{j|_{k}},x_{j|_{k}})\}\times  \big( [x_{j|_{k}}-\delta,x_{j|_{k}}+\delta]\times (\gamma,\gamma^{-1})\big)\Big)=0.
\end{equation} 
Since $m_{1}$ is simple, using (\ref{finite1}) we can choose  an $\epsilon<\delta/2$ small enough so that $\bar{B}_{\epsilon}(t_1,x_1),\dots, \bar{B}_{\epsilon}(t_{q_1},x_{q_1})$ are disjoint subsets of $(0,t')\times (\gamma,\gamma^{-1})$. Similarly, using (\ref{finitek}) for $\epsilon>0$ small enough we have, for $k=1,\dots,l-1$ and for any $j_1=1,\dots,q_1$; $j_2=1,\dots,q_2(j_1);\dots;$ $j_{k}=1,\dots,q_{k}(j|_{k-1})$,
\begin{equation}\label{disjointness}
\left\{\bar B_{\epsilon}(t_{j|_{k+1}},x_{j|_{k+1}}):j_{k+1}=1,\dots,q_{k+1}(j|_k)\right\}
\end{equation}
is a disjoint collection of subsets in $(0,x_{j|_k}-\delta)\times (\gamma,\gamma^{-1})$. 

Via $m_{1}(\{(t')\}\times (\gamma,\gamma^{-1}))=0$ and (\ref{finite1}), we can use Proposition\ref{gozde} to get
\begin{equation}\label{cont1}
m_1^\nn(\cdot\; \cap (0,t']\times \gam)=\sum_{j_1=1}^{q_1}\vep_{(t_{j_1}^\nn,x_{j_1}^\nn)}
\end{equation}
and
\begin{equation}\label{cont2}
(t_{j_1}^\nn,x_{j_1}^\nn)\in \bar{B}_{\epsilon}(t_{j_1},x_{j_1}),\quad \forall j_1=1,\dots,q_1.
\end{equation}
Similarly, using (\ref{finitek}), (\ref{noboundary}) and that (\ref{disjointness}) is a disjoint collection of subsets in $(0,x_{j|_k}-\delta)\times (\gamma,\gamma^{-1})$, we can employ Proposition \ref{gozde} to conclude the following: for any $k=1,\dots,l-1$ and for any $j_1=1,\dots,q_1$; $j_2=1,\dots,q_2(j_1);\dots;$ $j_{k}=1,\dots,q_{k}(j|_{k-1})$,
\begin{equation}\label{cont3}
\begin{aligned}
&\;\;m^\nn_{k+1}\Big(\cdot\;\cap \;\bar B_{\epsilon}(t_{j_1},x_{j_1})\times \cdots\times \bar B_{\epsilon}(t_{j|_k},x_{j|_k})\times (0,x_{j|_k}]\times \gam\Big)
\\&=m^\nn_{k+1}\Big(\cdot\;\cap \;\bar B_{\epsilon}(t_{j_1},x_{j_1})\times \cdots\times \bar B_{\epsilon}(t_{j|_k},x_{j|_k})\times (0,x^\nn_{j|_k}]\times \gam\Big)
\\&=\sum_{j_{k+1}=1}^{q_{k+1}(j|_k)}\vep_{(t^\nn_{j_1},x^\nn_{j_1},\dots,t^\nn_{j|_{k+1}},x^\nn_{j|_{k+1}})},
\end{aligned}
\end{equation}
and
\begin{equation}\label{cont4}
(t_{j|_{k+1}}^\nn,x_{j|_{k+1}}^\nn)\in \bar B_{\epsilon }(t_{j|_{k+1}},x_{j|_{k+1}}),\quad \forall j_{k+1}=1,\dots, q_{k+1}(j|_k).
\end{equation}
Let $f,g\in D([0,t'])$. Designating by $d(f,g)$ the $J_1$ distance on $D([0,t'])$, recall that
\begin{equation}\label{cont5}
d(f,g)=\inf_{\lambda\in \Lambda}\left\{\sup_{t\in[0,t']}\parallel \lambda(t)-t\parallel \vee\sup_{t\in[0,t']}\parallel f(\lambda(t))-g(t) \parallel\right\},
\end{equation}
where the $\Lambda$ is the set of strictly increasing, continuous mappings of $[0,t']$ onto itself and $\vee$ stands for maximum.

Let $\lambda^\nn:[0,t']\to[0,t']$ be the piecewise linear map that maps $t_{j_1}$ to $t_{j_1}^\nn$ for all $j_1=1,\dots,q_1$ with $\lambda^\nn(0)=0$ and $\lambda^\nn(t')=t'$. Then by (\ref{cont1}) and (\ref{cont2}), we have for all $n$ large enough,
\begin{equation}
\sup_{t\in[0,t']}\parallel \lambda^\nn(t)-t \parallel \leq \epsilon q_1.
\end{equation}
Moreover, by (\ref{cont3}) and (\ref{cont4}), for all $n$ large enough,
\begin{equation}
\sup_{t\in[0,t']}\parallel T^\g_l(m^\nn)(\lambda^\nn(t))-T^\g_l(m) \parallel\leq \epsilon \sum_{j_1=1}^{q_1}\cdots \sum_{j_l=1}^{q_{l}(j|_{l-1})}1.
\end{equation}
Hence, by (\ref{cont5}), we are finished by the proof of continuity of $T_l^\g$.

Next, we prove that $\ovl{T}^\g_l(m^\nn)\to \ovl{T}^\g_l(m)$. We use the notation $\ovl{T}_l^\g(m)=\ovl{m}$ with $\ovl{m}=(\ovl{m}_1,\dots,\ovl{m}_{l-1})$ and similarly, $\ovl{T}_l(m^\nn)=\ovl{m}^\nn$ with $\ovl{m}^\nn=(\ovl{m}^\nn_1,\dots,\ovl{m}^\nn_{l-1})$. We first prove that $\ovl{m}_1^\nn\to \ovl{m}_1$. It is enough to show that for any $0<s_1<s_2$ and $F\subset (0,\infty)$ relatively compact with  $\ovl{m}_1(\partial ([s_1,s_2]\times F))=0$,  $\ovl{m}^\nn_1([s_1,s_2]\times F)\to\ovl{m}^\nn_1([s_1,s_2]\times F)$. Let $\{(t_{j_1}^\g,x_{j_1}^\g):j_1\in\N\}$ denote the marks of $m_1^\g$ labeled so that $t_1^\g<t_2^g<\cdots$ and let $p=\inf\{m:x^\g_1+\cdots +x_m^\g>t\}+1$. Let $t_p<t'<t_{p+1}$. Hence, after relabeling,
\begin{equation}
m_1((0,t']\times (\gamma,\gamma^{-1}))=\sum_{j_1=1}^p \vep_{(t_{j_1},x_{j_1})}
\end{equation}
and $m_1(\partial \big((0,t']\times \gam\big))=0$.
We choose $\delta>0$ small enough as in (\ref{noboundary}). As before, we choose $\epsilon<\delta/2$ small enough so that $\bar{B}_{\epsilon}(t^\g_{1},x^\g_{1}),\dots,\bar{B}_{\epsilon}(t^\g_{p},x^\g_{p})\subset (0,t']\times\gam$ is a disjoint collection.
Using Proposition \ref{gozde} we have for any $\epsilon>0$, after relabeling, for $n$ large enough,
\begin{equation}\label{anani1}
m_1^\nn(\cdot\;\cap\;  (0,t']\times \gam)=\sum_{j_1=1}^{p} \vep_{(t_{j_1}^\nn,x_{j_1}^\nn)}
\end{equation}
and
\begin{equation}
(t_{j_1}^\nn,x_{j_1}^\nn)\in \bar{B}_{\epsilon}(t_{j_1},x_{j_1}),\quad \forall j_1=1,\dots,p.
\end{equation}
Hence, using Proposition (\ref{gozde}) we get for all $j_1=1,\dots,p$,
\begin{equation}
m_2(\bar{B}_{\epsilon}(t_{j_1},x_{j_1})\times \bigl( (0,x^\g_{j_1}]\times \gam\bigr))=m_2^\nn(\bar{B}_{\epsilon}(t^\nn_{j_1},x^\nn_{j_1})\times \bigl((0,x_{j_1}^\nn]\times \gam\bigr)).
\end{equation}
Finally, choosing $\epsilon$ small enough so that $2 p \epsilon$ is smaller than the distance of the marks of $\ovl{m}_1$ in $[s_1,s_2]\times F$ to the boundary of $[s_1,s_2]\times F$ proves that $\ovl{m}^\nn_1([s_1,s_2]\times F)\to\ovl{m}^\nn_1([s_1,s_2]\times F)$. The convergence of $\ovl{m}_k^\nn$ to $\ovl{m}_k$ is trivial using the same exact proof and definitions.

\end{proof}

\subsection{Proof of the convergence of the clock processes} 
\begin{proof}[Proof of Theorem \ref{fixL}]

\newcommand{\kk}{{(k)}}

Let $\chi^\lag$ be a RPC with parameters $\alpha_1,\dots,\alpha_{\lag}$ and constants\\ $D_1,\dots,D_\lag$, where the latter collection is given by Lemma \ref{lemanara}. Note that a.s. $\chi^{\lag}\in \M_{1,\lag}^\g$. As before, we define $\chi^{\g,\lag}=(\chi_{1}^{\g,\lag},\dots,\chi_{\lag}^{\g,\lag})$ by
\begin{equation}
 \chi_{k}^{\g,\lag}=\sum_{j|_k\in\N^k}\vep_{(t_{j_1},\chi_{j_1},\dots,t_{j|_k},\chi_{j|_k})}\prod_{i=1}^k\mathds{1}\left\{\chi_{j|_i}\in\gam\right\}.
\end{equation}
For $k=1,\dots,\lag$, we define the maps $T_{k,\lag}^\g:\M_{1,\lag}\to D([0,\infty))$ by
\begin{equation}
T_{k,\lag}^\g=T_{\lag-k+1}^\g\bigl(\overline{T}^\g_{\lag-k+2}\circ\cdots\circ \overline{T}^\g_{\lag}\bigr).
\end{equation}
We also set $T_{k,\lag}^\g$ as the identity map for any $k>\lag$. 
\newcommand{\kp}{^{(k)}}

Let $(\Upsilon_k:k=1,\dots,\lag)$ be an independent collection of PPPs on $H$ where for each $k$, the mean measure of $\Upsilon_k$ is $dt\times D_k\alpha_k x^{-1-\alpha_k}dx$. We denote the marks of $\Upsilon_k$ by $\Upsilon_k=\sum_{i\in\N}\vep_{(t_i\kp,\eta_i^{(k)})}$. For $\gamma>0$, set $\Upsilon_k^\g=\sum_{i\in\N}\vep_{(t_i\kp,\eta_i^{(k)})}\mathds{1}\{\eta_i^{(k)}\in(\gamma,\gamma^{-1})\}$. Due to correlation structure of $\chi^{\lag}$, we have $\ovl{T}_2^\g\circ \cdots \circ \overline{T}_\lag^\g(\chi^\lag)$ is independent of $(\chi_{1}^{\lag},\dots,\chi_{\lag-1}^{\lag})$ and has the same distribution as $\Upsilon_{\lag}^\g$. Using this, the fact that $(\chi_{1}^{\lag},\dots\chi_{\lag-1}^{\lag})\overset{d}=\chi^{\lag-1}$ where  $\chi^{\lag-1}$ has the distribution of a PRC with parameters $\alpha_1,\dots,\alpha_{\lag-1}$ and constants $D_1,\dots,D_{\lag-1}$, and the identity (\ref{identity}) we get
\begin{equation}
\begin{aligned}
&\left(T^\g_{k,\lag}(\chi^\lag):k=1,\dots,\lag\right)
\\&=\left(T^\g_{\lag,\lag}(\chi^\lag)\circ T^\g_{k,\lag-1}(\chi_{1}^{\lag},\dots,\chi_{\lag-1}^{\lag}):k=1,\dots,\lag\right)
\\&=\left(T_1^\g(\ovl{T}^\g_2\circ\cdots\circ\ovl{T}^\g_\lag)(\chi^\lag)\circ T^\g_{k,\lag-1}(\chi_{1}^{\lag-1},\dots,\chi_{\lag-1}^{\lag}):k=1,\dots,\lag\right)
\\&\overset{d}=\left(T_1^\g(\Upsilon_\lag^\g)\circ T^\g_{k,\lag-1}(\chi^{\lag-1}):k=1,\dots,\lag\right),
\end{aligned}
\end{equation}
where, in the last display above $\Upsilon_{\lag}^\g$ and $\chi^{\lag-1}$ are independent. Proceeding inductively, we reach
\begin{equation}\label{yettigari}
\left(T^\g_{k,\lag}(\chi^\lag):k=1,\dots,\lag\right)=\left(T_1^\g(\Upsilon_\lag^{\g})\circ\cdots\circ T_{1}^\g(\Upsilon_k^{\g}):k=1,\dots,\lag\right).
\end{equation}
Let $V_{\alpha_k}(t)=\sum_{t_i\kp\leq t}\eta_i^\kk$ and $V_{\alpha_k}^\g(t)=\sum_{t_i\kp\leq t}\eta_i^\kk\mathds{1}\{\eta_i^\kk\in (\gamma,\gamma^{-1})\}$. Note that, $V_{\alpha_k}$ and $V_{\alpha_k}^\g$ are L\'evy subordinators with corresponding L\'evy measures $D_k\alpha_kx^{-1-\alpha_k}dx$ and $\mathds{1}\{x\in(\gamma,\gamma^{-1})\}D_k\alpha_kx^{-1-\alpha_k}dx$, respectively. We set
$
 \widetilde{V}_{k,\lag}=V_{\alpha_{\lag}}\circ V_{\alpha_{\lag-1}}\circ\cdots\circ V_{\alpha_k},
$
and
$
  \widetilde{V}^\g_{k,\lag}=V^\g_{\alpha_{\lag}}\circ V^\g_{\alpha_{\lag-1}}\circ\cdots\circ V^\g_{\alpha_k}.
$
Hence, by (\ref{yettigari}), we get 
\begin{equation}\label{susa1}
 \Big(T_{k,\lag}^\g(\zeta^\lag):k=1,\dots,\lag \Big)\overset{d}= \Big(\widetilde{V}^\g_{k,\lag}:k=1,\dots,\lag\Big).
\end{equation}
It is trivial that as $\gamma\to 0$
\begin{equation}\label{vs2}
 \Big(\widetilde{V}^\g_{k,\lag}:k=1,\dots,\lag\Big)\Longrightarrow \Big(\widetilde{V}_{k,\lag}:k=1,\dots,\lag\Big)
\end{equation}
weakly on the space $D^{\lag}([0,\infty))$ equipped with the product Skorohod $J_1$ topology.

By Lemma \ref{contlem}, for each $k$, $T^\g_{k,\lag}$ is continuous on $\M_{1,\lag}^\g$. By Proposition \ref{weaklemma} $\zeta^{\nn,\lag}\Rightarrow \chi^\lag$ for every environment in $\Omega'_L$, where $\mathbb{P}(\Omega'_L)=1$. Finally, since a.s. $\chi^{\lag}\in \M_{1,\lag}^\g$ for any $\gamma>0$  we can employ the continuous mapping theorem (see Theorem 2.7 in \cite{Bill}) to conclude that for any $\gamma>0$
\begin{equation}
 \Big(T_{k,\lag}^\g(\zeta^{\nn,\lag}):k=1,\dots,\lag\Big)\Longrightarrow \Big(\widetilde{V}^\g_{k,\lag}:k=1,\dots,\lag\Big).
\end{equation}
\iffalse
Now we define the clock processes where the jumps are restricted to $(\gamma,\gamma^{-1})$, after rescaling. More precisely, defining $\zeta_{\lag}^{\nn,\g,L}=(\zeta_{1,\lag}^{\nn,\g,L},\dots,\zeta_{\lag,\lag}^{\nn,\g,L})$ by
\begin{equation}\label{rcollg}
 \zeta_{k,\lag}^{\nn,\g,L}=\sum_{j_1,\dots,j_k\in\N}\vep_{(t_{j_1}^\nn,\xi_{j_1}^\nn,\dots,t_{j_1\dots j_k}^\nn,\xi_{j_1\dots j_k}^\nn)}\prod_{i=1}^k \mathds{1}\{\xi_{j|_i}^\nn\in (\gamma,\gamma^{-1})\},
\end{equation}
we set
\begin{equation}
S_{k,L}^{\nn,\g}=T_{\lag-k+1}\bigl(\widetilde{T}_{\lag-k+2}\circ\cdots \tilde{T}_\lag\bigr)(\zeta_\lag^{\nn,\g,L}).
\end{equation}
\fi
Now we want to prove that, $\mathbb{P}$-a.s. for any $k=1,\dots,\lag$ and $\epsilon>0$
\begin{equation}
 \lim_{\gamma\to0}\lim_{n\to \infty} \Pp\Big(d\big(S_{k,L}^\nn,T_{k,\lag}^\g(\zeta^{\nn,\lag})\big)\geq \epsilon\Big)=0
\end{equation}
where $d$ denotes the $J_1$ distance on $D([0,\infty))$. It is enough to check that $\mathbb{P}$-a.s. for any $k=1,\dots,\lag$ and $t>0$,
\begin{equation}\label{aaa}
\lim_{\gamma\to 0}\lim_{n\to\infty}\Pp(A_{k,\lag}^{\nn,\g}(t)\geq \epsilon)=0 
\end{equation}
where for $k_1=1,\dots,\lag$ and $k_2=k_1,\dots,\lag$,
\begin{equation*}
 A_{k_1,k_2}^{\nn,\gamma}(t)=\sum_{t^\nn_{j_1}\leq t } \;\; \sum_{t_{j|_2}\leq \xi_{j_1}^\nn  }\cdots \sum_{t_{j|_{k_2}}\leq \xi^\nn_{j|_{k_2-1}}} \xi_{j|_{k_2}}^\nn \inc{\xi^\nn_{j|_{k_1}}\notin (\gamma,\gamma^{-1}) }.
\end{equation*}
Then, using Lemma \ref{lemanara} and the conditioning argument in the proof of Proposition \ref{weaklemma} we can conclude that $\exists \ovl{\Omega}_L$ with $\mathbb{P}(\ovl\Omega_L)=1$ such that for any environment in $\ovl\Omega_L$, for any $k=1,\dots,\lag$ and $t>0$,
\begin{equation*}
A_{k,\lag}^{\nn,\g}(t)\Rightarrow V_{\alpha_{\lag}}\circ\cdots\circ V_{\alpha_{k+1}}\circ V_{\alpha_{k},\g}\circ V_{\alpha_{k-1}}\circ\cdots\circ V_{\alpha_1}(t),
\end{equation*}
where $V_{\alpha_1},\dots,V_{\alpha_{k-1}},V_{\alpha_{k+1}},\dots V_{\alpha_{\lag}}$ are as before and $V_{\alpha_k,\g}(t)=\sum_{t_i\leq t}\eta_i^\kk\mathds{1}\{\eta_i^\kk\notin (\gamma,\gamma^{-1})\}$. Thus,
\begin{equation}\label{weakA}
A_k^{\nn,\gamma}(t)\Rightarrow V_{\alpha_{k+1}\alpha_{k+2}\cdots\alpha_{\lag}}\circ V_{\alpha_k,\g}\circ V_{\alpha_{1}\cdots\alpha_{k-1}}(t),
\end{equation}
where $V_{\alpha_{k+1}\alpha_{k+2}\cdots\alpha_{\lag}}$ and $V_{\alpha_{1}\cdots\alpha_{k-1}}$ are two independent subordinators that are $\alpha_{k+1}\alpha_{k+2}$ $\cdots\alpha_{\lag}$ and ${\alpha_{1}\cdots\alpha_{k-1}}$ stable, respectively, and independent from $V_{\alpha_k,\g}$.
For any $0<\alpha,\alpha'<1$, let $V_{\alpha,\g}$ and $V_{\alpha'}$ be independent subordinators with L\'evy measures $\nu_{\g}(dx)=\mathds{1}\{x\notin (\gamma,\gamma^{-1})\}D\alpha x^{-1-\alpha}dx$ and $\nu'(dx)=D'\alpha'x^{-1-\alpha'}dx$, respectively. For any $T>0$ we have $E[V_{\alpha,\g}(T)]\leq C T \gamma^{1-\alpha}$ for some positive constant $C$. For any $\epsilon'>0$ given, if we choose $T$ large enough so that $P(V_{\alpha'}(t)>T)\leq \epsilon'$, by Chebyshev inequality we get
\begin{equation}\label{es1}
P(V_{\alpha,\g}(V_{\alpha'})(t)\geq \epsilon)\leq \frac{CT\gamma^{1-\alpha}}{\epsilon}+\epsilon'.
\end{equation}
Moreover, by the stability of $\alpha$-stable subordinators we have $P(V_{\alpha'}(V_{\alpha,\g})(t)\geq \epsilon)=P(V_{\alpha'}(1)(V_{\alpha,\g}(t))^{1/\alpha}\geq \epsilon)$. Hence, if we choose now $M$ large enough so that $P(V_{\alpha'}(1)\geq M)\leq \epsilon'$ we get
\begin{equation}\label{es2}
P(V_{\alpha'}(V_{\alpha,\g})(t)\geq \epsilon)\leq \frac{M^\alpha t \gamma^{1-\delta}}{\epsilon^{\alpha}}+\epsilon'.
\end{equation}
Hence, combining (\ref{es1}) and (\ref{es2}) with (\ref{weakA}) we get (\ref{aaa}).
Thus, using (\ref{vs2}) and setting $\wt\Omega_L=\Omega'_L\cap \ovl{\Omega}_L$, we have $\mathbb{P}(\wt\Omega_L)=1$ and for any environment in $\wt\Omega_L$, as $n\to\infty$
\begin{equation}
 \Big(S_{k,L}^\nn:k=1,\dots,\lag\Big)\Longrightarrow \Big(\widetilde{V}_{k,\lag}:k=1,\dots,\lag\Big).
\end{equation}
Recall (\ref{Ralphas}) and that $\alpha_{k,L}=\alpha_k$. Observe that,
\begin{equation}
 Z_{R(k-1/L),R(k/L)}(c_k\cdot)\overset{d}=V_{\alpha_k}(\cdot),\quad c_k=D_{k}{\Gamma(1-\alpha_k)}.
\end{equation}
The independence of $Z_{R(k-1/L),R(k/L)}$ in $k$ and the stability property finish the proof of Theorem \ref{fixL} with the constants $b_{\lag,L}=c_{\lag}$ and
$
 b_{k,L}=c_k c_{k+1}^{\alpha_{k}} c_{k+2}^{\alpha_{k}\alpha_{k+1}}\cdots c_{\lag}^{\alpha_{k}\cdots\alpha_{\lag-1}},
$ for $k=1,\dots,\lag-1$.

\end{proof}

\section{Convergence of the two-time correlation function.}\label{Section3}

\begin{proof}[Proof of Theorem \ref{thrm1}]
Observe that $\{S_{k,L}(i):i\in\N\}\cap [t,t+s]=\emptyset$ implies that $\langle X(t),X(t+u)\rangle \geq k,\;\forall u\in[0,s]$. On the other hand, if $\{S_{k,L}(i):i\in\N\}\cap [t,t+s]\not=\emptyset$ there is at least one jump beyond and including the $k$-th level during the time interval $[t,t+s]$. Thus, the only way $\langle X(t),X(t+u)\rangle \geq k,\;\forall u\in[0,s]$ can happen is that at each such jump the discrete Markov chain $Y|_k$ jumps back to the same vertex it jumped from, and since $Y|_k$ chain chooses the last coordinate $\mu_k$ uniform at random on $[n]$, we arrive at
\begin{equation}\label{jump2}
\Pi_{k}(t,s)={P}\Big(\{S_{k,L}(i):\;i\in\N\}\cap[t,t+s]=\emptyset\Big)+O(1/n).
\end{equation}
By Theorem \ref{fixL}, for $k=1,\dots,\lag$, for any environment in $\wt\Omega_L$, as $n\to\infty$
\begin{equation}
S_{k,L}^\nn(\cdot)\Longrightarrow \tilde{Z}_{k,\lag}(\cdot)
\end{equation}
on $D([0,\infty))$ equipped with $J_1$ topology, where $\tilde{Z}_{k,\lag}$ is as in Theorem \ref{fixL}. Since $\tilde{Z}_{k,\lag}$ is an $\alpha$-stable subordinator with $\alpha=\alpha_{k,L}\cdots \alpha_{\lag,L}$, and thus, its L\'evy measure has no atoms. Hence, using the continuous mapping theorem and (\ref{jump2}), we can conclude that for any environment in $\wt\Omega_L$, as $n\to\infty$
\begin{equation}
\Pi_k^\nn(\tn,\theta\tn)\longrightarrow {P}(\{\tilde{Z}_{k-1,\lag}(r):r\geq 0\}\cap [1,1+\theta]=\emptyset)=Asl_{\alpha_{k,L}\cdots \alpha_{\lag,L}}\left(\frac{1}{1+\theta}\right).
\end{equation}

We now prove that $\mathbb{P}$-a.s. $\Pi_{\lag+1}(\tn,\theta\tn)\to 0$ as $n\to\infty$. This convergence implies that $\Pi_{k}(\tn,\theta\tn)\to 0$ for $k\geq\lag+1$, since  $\Pi_{k+1}(\tn,\theta\tn)\leq \Pi_k(\tn,\theta\tn)$, and finishes the proof. Let $S_{\lag,\lag}=T_{\lag,\lag}(m)=T_{\lag,\lag}(m|_\lag)$. Recalling (\ref{impcomp}), we have
\begin{equation}\label{comps1}
S_{\lag+1,L}\circ{S}_{\lag,\lag}=S_{\lag,L}.
\end{equation}
Observe that, $S_{\lag,\lag}$ corresponds to the $\lag$-th clock process of the GREM-like trap model on $\lag$-levels tree where the environment is kept the same on those levels, the only difference is that it is in discrete time where the exponential waiting times are replaced by geometric random variables. Also, note that $c_n(\lag)$ is a time scale of observation for this GREM-like trap model where all the levels are aging.  Therefore, with the same exact proof of Theorem \ref{fixL} we can see that $\exists \Omega''_L$ with $\mathbb{P}(\Omega''_L)=1$ such that for any environment in $\Omega''_L$, as $n\to\infty$ 
\begin{equation}\label{comps2}
\frac{S_{\lag,\lag}(\cdot\; \aan{\lag})}{\ccn{\lag}}\Longrightarrow V_{\alpha_{\lag,L}}(\cdot)
\end{equation}
weakly on $D([0,\infty))$ equipped with Skorohod $J_1$ topology, where $V_{\alpha_{\lag,L}}$ is an $\alpha_{\lag}$-stable subordinator. Using Theorem \ref{fixL}, for any $\epsilon>0$ we can choose $t$ large enough so that for all $n$ large,  
\begin{equation}
\Pp\Bigl(\frac{S_{\lag,L}(t\aan{\lag})}{\tn}>1+2\theta\Bigr) \geq 1-\epsilon/4.
\end{equation}
For $t$ chosen as above, using (\ref{comps2}) we can choose $T>0$ large enough so that for all $n$ large,
\begin{equation}
\Pp\Bigl(\frac{S_{\lag,\lag}(t \aan{\lag})}{\ccn{\lag}}\leq T\Bigr)\geq 1-\epsilon/4.
\end{equation} 
Hence, by (\ref{comps2}) we have for all $n$ large (recall that $c_n(\lag)=a_n(\lag+1)$),
\begin{equation}\label{critic}
\mathcal{P}\Big(\frac{S_{\lag+1,L}(T\aan{\lag+1})}{\tn}\geq 1+2\theta\Big)\geq 1-\epsilon/2.
\end{equation}
Recall the collection in (\ref{colljumps}) and how we have constructed it from the chain $J$. For $r=\lag+1,\dots,L$, let
\begin{equation}
S_{1,L,\g}^{r}(t)=\sum_{j_1=1}^{t}\;\;\sum_{{j_2}=1}^{\xi_{j_1}}\cdots\sum_{{j_{\lag+1} }=1}^{\xi_{j|_{\lag}}}\Lambda^r_{\g}(j|_{\lag+1})
\end{equation}
where  
\begin{equation}
 \Lambda^{r}_\g(j|_{\lag+1})=\sum_{j_{\lag+2}=1}^{\xi_{j|_{\lag+1}}} \cdots \sum_{j_{L}=1}^{\xi_{j|_{L-1}}}\xi_{j|_L}\mathds{1}\{\lambda^{-1}(J({j|_r}))\leq \gamma n^{1/\alpha_r}\}.
\end{equation}
Here, $S_{1,L,\g}^{r}$ is the 1st level clock process, where, restricted to jumps where the walk is at a vertex $\mu$ with $\lambda^{-1}(\mu|_r)\leq \gamma n^{1/\alpha_r}$. A modification of the proof of Lemma \ref{propcasc} yields that $\mathbb{P}$-a.s. for any $t>0$, as $n\to\infty$
\begin{equation}
S_{1,L,\g}^{r}(ta_n(1))/{c_n}\Rightarrow S^{r}_\g(t)
\end{equation}
where $S_\g^{r}(t)$ is a positive random variable with the Laplace transform
\begin{equation}
E[e^{-\kappa S_\g^{r}(t)}]=e^{-\phi_\g(t,\kappa)}, \quad \kappa>0,
\end{equation}
and for any $\kappa>0$ and $t>0$, $\phi_\g(t,\kappa)\to0$ as $\gamma\to 0$. Hence, $\mathbb{P}$-a.s.~for any $r=\lag+1,\dots,L$, $t>0$ and $\epsilon>0$
\begin{equation}\label{critic2}
\lim_{\gamma\to 0}\limsup_{n\to\infty}\Pp\left(S_{1,L,\g}^{r}\big(ta_n(1)\big)\geq \epsilon c_n\right)=0.
\end{equation}
For $T'>0$, define
\begin{equation}
\begin{aligned}
A_n=&\left\{\exists j_1=1,\dots,T'\aan{1},\exists j_2=1,\dots,T'\aan{2},\dots,\exists j_{\lag+1}=1,\dots,T'\aan{\lag+1} \text{ s.t. } \right.\\&\;\;\;\;\left.\forall k=1,\dots,\lag \;\xi^\nn(j|_k)\geq \gamma\;\text{ and } \Lambda^\g(j|_{\lag+1})\geq \theta\tn/2 \right\}
\end{aligned}
\end{equation}
where
\begin{equation}
\Lambda^{\g}(j|_{\lag+1})=\sum_{j_{\lag+2}=1}^{\eta_{j|_{\lag+1}}} \cdots \sum_{j_{L}=1}^{\eta_{j|_{L-1}}}\eta_{j|_L}\prod_{r=\lag+1}^L \mathds{1}\{\lambda^{-1}(J({j|_r}))> \gamma n^{1/\alpha_r}\}.
\end{equation}
By, (\ref{aaa}), (\ref{critic}) and (\ref{critic2}) we have $\mathbb{P}$-a.s. for $T'$ large enough, for all $n$ large and $\gamma$ small 
$\Pi_{\lag+1}^\nn(t\tn,s\tn)\leq 
\Pp(A_n)+\epsilon/2.
$ Recall that $\tn=n^{1/\alpha_L+c}$, for some $c>0$. We define 
\begin{equation}
\begin{aligned}
W^\nn=&\left\{\mu|_\lag\in V|_\lag:\;\text{ if } \mu_{\lag+1},\dots,\mu_L\text{ s.t. } \forall r=\lag+1,\dots,L-1\right.
\\&\left. \lambda^{-1}(\mu|_\lag\mu_{\lag+1}\cdots\mu_r)\geq \gamma n^{1/\alpha_r} \text{ then } \max_{\mu_L\in [n]}\lambda^{-1}(\mu|_L)\leq n^{1/\alpha_L+c/2}\right\},
\end{aligned}
\end{equation}
and
\begin{equation}
\begin{aligned}
B_n=&\left\{\exists j_1=1,\dots,T'\aan{1},\dots,\exists j_{\lag}=1,\dots,T'\aan{\lag} \text{ s.t. } \right.\\&\left.\forall k=1,\dots,\lag\;\xi^\nn(j|_k)\geq \gamma\; \text{ and } J(j|_{\lag})\notin W^\nn  \right\}
\end{aligned}
\end{equation}
Since the maximum of $n^d$ i.i.d.~mean one exponential random variables is of order $\log n$, it follows that $\mathbb{P}$-a.s. for all $n$ large enough $\Pp(A_n)\leq \Pp(B_n)+\epsilon/2$. Note that
\begin{equation}\label{sokolo}
\mathbb{P}(\mu|_{\lag}\notin W^\nn)\leq Cn^{-\alpha_Lc/2},
\end{equation}
for some $C>0$. Since under $\mathbb{P}$, $\mathds{1}\{\mu|_{\lag}\notin W^\nn\}$ is independent of the random environment on the top $\lag$ levels we have
\begin{equation}
\mathbb{E}\left[\Pp(B_n)\right]=\left(\prod_{k=1}^\lag \aan{k}\mathbb{E}\left[\Pp\bigl(G(\lambda(\mu|_k)\geq \gamma \ccn{k}\bigr)\right]\right)\mathbb{P}(\mu|_{\lag}\notin W^\nn).
\end{equation}
By Lemma \ref{lemanara} and (\ref{sokolo}) the above quantity converges to 0 as $n\to\infty$. Since $\aan{1}\ll n$, we can control the fluctuations as in the proof of Lemma \ref{lemanara} to conclude that $\mathbb{P}$-a.s. as $n\to \infty$, $\Pp(B_n)\to 0$. Hence, denoting by $\ovl{\Omega}_L$ the subset of $\Omega$ such that for any environment in $\ovl{\Omega}_L$, as $n\to\infty$, $\Pp(B_n)\to 0$, we are finished with the proof of Theorem \ref{thrm1} where $\Omega_L=\wt\Omega_L\cap \ovl{\Omega}_L$.
\end{proof}

\begin{proof}[Proof of Corollary \ref{corol1}]
Since $R$ is uniformly continuous, for any $\epsilon>0$ given, we have for all $L$ large enough
\begin{equation}\label{sonispat1}
0\leq R\bigl(k/L\bigr)-R\bigl((k-1)/L\bigr)\leq \epsilon\quad \forall k=1,\dots,L.
\end{equation}
Recall the definition of $\alpha_{k,L}$
\begin{equation}
 \alpha_\kl{k}=\exp\Big\{-\Big(R\big(k/L\big)-R\big(k/L-1/L\big)\Big)\Big\}.
\end{equation}
Using (\ref{sonispat1}) for all $k=1,\dots,L$, 
\begin{equation*}
0\leq e^{-\epsilon}\bigl(R\bigl(k/L\bigr)-R\bigl((k-1)/L\bigr)\bigr)\leq 1-\alpha_{k,L}\leq R\bigl(k/L\bigr)-R\bigl((k-1)/L\bigr)\leq \epsilon. 
\end{equation*}
Hence, using the definition of $d_{k,L}$ (see (\ref{crtscales})), for all $k=1,\dots,L$, we have
\begin{equation}
R\bigl(1\bigr)-R\bigl((k-1)/L\bigr) +1\leq d_{k,L}\leq e^{\epsilon} \bigl(R\bigl(1\bigr)-R\bigl((k-1)/L\bigr)\bigr) +1.
\end{equation}
As a consequence,
\begin{equation}\label{sonispat2}
\lim_{L\to\infty}\frac{\lag(\rho)}{L}=r^*(\rho)=\sup\{s\geq 0: R(1)-R(s)+1>\rho\}.
\end{equation}
Since $\rho (0,d_{1,L})\setminus \{d_{L,L},\dots, d_{2,L}\}$ for all $L$ large enough, using Theorem \ref{thrm1}, we can conclude that for any $t,s>0$, for any environment in $\Omega'=\cap_L\Omega_L$, for any $L$ large enough, (\ref{corrlimit}) is satisfied. Via (\ref{sonispat2}) and the uniform continuity of $R$ we get
\begin{equation}
\bar{\alpha}_k=\exp\left(\bigl(R\bigl(\lag(\rho)/L\bigr)-R\bigl((k-1)/L\bigr)\bigr)\right)=\exp(\left(-(R\bigl(r^*(\rho)\bigl)-R\bigl(k/L\bigr)\right)+c(\epsilon)
\end{equation}
where $c(\epsilon)\to 0$ as $\epsilon \to 0$. Finally, since $q$ is smooth, using (\ref{sonispat2}) one last time, we see that the left hand side of (\ref{corrlimit}) is the Riemann sum approximation of the integral on the left hand side of (\ref{eqcorol1}). Thus, we are finished with the proof. 
\end{proof}
\bibliographystyle{plain}
\bibliography{GREMTrapBibl1}

\end{document}